\documentclass[10pt,a4paper]{amsart}

\usepackage[english]{babel}
\usepackage[utf8]{inputenc}

\usepackage{amsmath}
\usepackage{amsthm}
\usepackage{amsfonts}
\usepackage{amssymb}
\usepackage{hyperref}
\usepackage{graphicx}
\usepackage{xcolor}
\usepackage{changes}

\newtheorem{theorem}{Theorem}[section]
\newtheorem{lemma}[theorem]{Lemma}
\newtheorem{proposition}[theorem]{Proposition}

\newtheorem{remark}[theorem]{Remark}

\numberwithin{equation}{section}

\newcommand*{\N}{\ensuremath{\mathbb{N}}}
\newcommand*{\Z}{\ensuremath{\mathbb{Z}}}
\newcommand*{\R}{\ensuremath{\mathbb{R}}}
\newcommand*{\C}{\ensuremath{\mathbb{C}}}
\newcommand*{\F}{\ensuremath{\mathcal{F}}}
\newcommand*{\e}{\ensuremath{\varepsilon}}

\author{Paul Bergold}
\address{(Paul Bergold) Department of Mathematics, University of Surrey, Guildford, UK}
\email{p.bergold@surrey.ac.uk}

\author{Caroline Lasser}
\address{(Caroline Lasser) Zentrum Mathematik, Technische Universit\"at M\"unchen, Germany}
\email{classer@ma.tum.de}

\date{\today}
\keywords{Schr\"odinger equation, Thawed Gaussian approximations, Gaussian wave packet transforms, Quadrature rules}
\subjclass[2010]{42A38, 65D32, 65P10, 65Z05, 81Q20}

\title[An error bound for the TSTG propagation method]{An error bound for the\\ time-sliced thawed Gaussian propagation method}

\begin{document}
\begin{abstract}
	We study the time-sliced thawed Gaussian propagation method, which was recently proposed for solving the time-dependent Schr\"odinger equation.
 	We introduce a triplet of quadrature-based analysis, synthesis and re-initialization operators to give a rigorous mathematical formulation of the method.
	Further, we derive combined error bounds for the discretization of the wave packet transform and the time-propagation of the thawed Gaussian basis functions.
	Numerical experiments in 1D illustrate the theoretical results.
\end{abstract}
\maketitle

\section{Introduction}
Algorithms for simulations of quantum dynamics play a central role in the field of numerical analysis since these methods nowadays are the computational keystone in many research areas such as quantum chemistry.
In this paper we consider the time-dependent Schr\"odinger equation
\begin{align}\label{eq:tdse}
	i\e\partial_t\psi(x,t)
	=-\frac{\e^2}{2}\Delta_x\psi(x,t)+V(x)\psi(x,t),\quad
	0<\e\ll 1,
\end{align}
where the function $V\colon\R^d\to\R$ is a smooth potential of sub-quadratic growth and the complex-valued wave function $\psi\colon\R^d\times\R\to\C$ depends on $x\in\R^d$ and $t\in\R$.
The right hand-side of \eqref{eq:tdse} is given by the action of the semiclassical operator
\begin{align*}
	H
	=H^\e
	:=-\frac{\e^2}{2}\Delta_x+V,
\end{align*}
as it results for example from the time-dependent Born--Oppenheimer approximation, where the small positive parameter $\e^2$ represents a mass ratio of nuclei and electrons, see e.g. \cite[chapter~II.2]{Lubich:2008}.
Since we assume that the potential is of sub-quadratic growth, $H$ is a self-adjoint linear operator on $L^2(\R^d)$ and therefore the spectral theorem provides the unitary propagator
\begin{align*}
	U(t)
	:=e^{-iHt/\e}
	\quad\text{for all $t\in\R$},
\end{align*}
which guarantees existence and uniqueness of the solution $\psi(t)=U(t)\psi_0$ for a given initial wave function $\psi_0\in L^2(\R^d)$.\\

Motivated by questions in physics and chemistry, various numerical algorithms for simulations of quantum dynamics have been developed during the last decades.
For example, reduced models via variational approximations have been investigated, which include the multi-configuration methods such as MCTDH, see \cite{Meyer:1990}, the variational multi-configuration Gaussian wave packet (vMCG) method \cite{Worth:2004}, or the variational Gaussian wave packets \cite{Heller:1976,Coalson:1990}.
Semiclassical approaches such as Hagedorn wave packets \cite{Gradinaru:2014,Faou:2009}, Gaussian beams \cite{Leung:2009,Zheng:2014,Liu:2013} or the Herman--Kluk propagator \cite{Herman:1984,Lasser:2017} have been developed to include quantum effects especially for high-dimensional systems, for which standard grid-based numerical methods are infeasible.\\

Recently, Kong~\emph{et al.} have proposed the time-sliced thawed Gaussian (TSTG) propagation method, see \cite{Kong:2016}, in which Gaussian wave packets are decomposed into linear combinations of Gaussian basis functions without the need of multidimensional numerical integration.
The resulting approximations of wave packets can be obtained by discretizing the inversion formula for the so-called FBI (Fourier-Bros-Iagolnitzer) transform, which is used in microlocal analysis to analyze the distribution of wave packets in position and momentum space simultaneously, see e.g. \cite{Martinez:2002}.
According to the FBI inversion formula, see e.g. \cite[Proposition~5.1]{Lasser:2020}, any square-integrable function $\psi\in L^2(\R^d)$ can be decomposed as
\begin{align}\label{eq:inverse_FBI}
	\psi
	=(2\pi\e)^{-d}\int_{\R^{2d}}\left\langle g_z\mid\psi\right\rangle g_z\,\mathrm{d}z,
\end{align}
where the inner product in $L^2(\R^d)$ is taken antilinear in its first and linear in its second argument and the semiclassically scaled wave packet $g_z\in\mathcal{S}(\R^d)$ is defined for a given Schwartz function $g\colon\R^d\to\C$ of unit norm, that could be but needn't be a Gaussian, and a phase space center $z=(q,p)\in\R^{2d}$ by
\begin{align}\label{eq:def_gz}
	g_z(x)
	:=\e^{-d/4} g\left(\frac{x-q}{\sqrt\e}\right)e^{ip\cdot(x-q)/\e},
	\quad x\in\R^d.
\end{align}
Wave packets of this form are typically used for numerical computations in quantum molecular dynamics and have been extensively studied in the literature, sometimes with different conventions for the phase factor, e.g. $e^{ip\cdot\left(x-q/2\right)/\e}$ in \cite[chapter~1.1.2]{Combescure:2012}.\\

A direct discretization of the integral in \eqref{eq:inverse_FBI} using a multivariate quadrature formula in phase space yields an approximation of the form
\begin{align}\label{eq:inverse_FBI_discrete}
	\psi
	\approx\sum_{\mathbf{k}\in\mathcal{K}}c_{\mathbf{k}}(\psi)\,g_{\mathbf{k}},
\end{align}
where $\mathcal{K}\subset\N^{2d}$ is a given finite multi-index set, e.g. a cube $\{\mathbf{k}\in\N^{2d}\,:\,k_j\le K\}$ or a simplex $\{\mathbf{k}\in\N^{2d}\,:\,\sum_{j=1}^{2d}k_j\le K\}$, the representation coefficients $c_{\mathbf{k}}(\psi)\in\C$ are complex numbers, depending on $\psi$ and the underlying discretization scheme, and the functions $g_{\mathbf{k}}:=g_{z_\mathbf{k}}$ are wave packets centered at the grid points $z_{\mathbf{k}}\in\R^{2d}$.
In particular, if both the represented function $\psi$ and the basis functions $g_{\mathbf{k}}$ are Gaussian wave packets, then the coefficients $c_{\mathbf{k}}(\psi)$, which for this case essentially sample inner products $\langle g_{\mathbf{k}}\mid\psi\rangle$ of two Gaussians, can be calculated using a formula for multidimensional Gaussian integrals, see Lemma~\ref{fact:inner}.
The choice of Gaussian functions is particularly attractive for time propagation, since the time-dependent Schr\"odinger equation with quadratic potential leaves the class of Gaussian wave packets invariant. 
This fact can be used to approximate the time-evolution of Gaussian wave packets in anharmonic potentials, a distinction being made as to whether the width matrix is chosen to be constant in time (frozen) or time-dependent (thawed) and we note that the wave packet transform \eqref{eq:inverse_FBI} has been used for different approximation schemes such as the Herman--Kluk propagator (frozen) or Gaussian beams (thawed).\\

The discretization in \eqref{eq:inverse_FBI_discrete} with Gaussian basis functions and uniform Riemann sums was used by Kong~\emph{et al.} and can be viewed as one of the main ingredients for the design of the TSTG method, which we investigate in the present paper.
\begin{remark}
	In the following we work with time-evolved basis functions and to emphasize that we distinguish between ``original'' and time-evolved basis functions, we write $g_{\mathbf{k},0}$ for the original and $g_{\mathbf{k}}(t)$ for the time-evolved basis functions.
\end{remark}
Starting from the representation of the initial wave function according to \eqref{eq:inverse_FBI_discrete}, the solution to the Schr\"odinger equation \eqref{eq:tdse} is approximated in the TSTG method after a short propagation time $\tau>0$ by the linear combination of time-evolved basis functions as
\begin{align*}
	\psi(\tau)
	=U(\tau)\psi_0
	\approx\sum_{\mathbf{k}\in\mathcal{K}}c_{\mathbf{k}}(\psi_0)\,U(\tau)g_{\mathbf{k},0}
	=\sum_{\mathbf{k}\in\mathcal{K}}c_{\mathbf{k}}(\psi_0)\,g_{\mathbf{k}}(\tau),
\end{align*}
where we introduced the abbreviations $\psi(\tau)$ and $g_{\mathbf{k}}(\tau)$ for $\psi(\bullet,\tau)$ and $g_{\mathbf{k}}(\bullet,\tau)$.
Using thawed Gaussians to approximate the time-evolution of each basis function, the discretization of the wave packet transform \eqref{eq:inverse_FBI_discrete} is brought into play again to represent the individual thawed Gaussian approximants $u_{\mathbf{k}}^\tau\approx g_{\mathbf{k}}(\tau)$ as
\begin{align*}
	u_{\mathbf{k}}^\tau
	\approx\sum_{\mathbf{k}'\in\mathcal{K}}c_{\mathbf{k'}}(u_{\mathbf{k}}^\tau)\,g_{\mathbf{k'},0},
\end{align*}
which enables to approximate the solution $\psi(\tau)$ directly in the original basis in terms of updated coefficients $c^{1,\tau}_{\mathbf{k}}$ as
\begin{align*}
	\psi(\tau)
	\approx\psi^{1,\tau}
	:=\sum_{\mathbf{k}\in\mathcal{K}}c^{1,\tau}_{\mathbf{k}}g_{\mathbf{k},0},
	\quad\text{where}\quad
	c^{1,\tau}_{\mathbf{k}}
	:=\sum_{\mathbf{k}'\in\mathcal{K}}c_{\mathbf{k'}}(\psi_0)c_{\mathbf{k}}(u_{\mathbf{k}'}^\tau).
\end{align*}
The concatenation of TSTG propagation steps then result in approximations for larger times $2\tau,3\tau,\dots$, which are obtained (without additional time-integration) by computing update coefficients $c^{2,\tau}_{\mathbf{k}},c^{3,\tau}_{\mathbf{k}},\dots$ of higher order.
Since all these coefficients have analytic representations, multidimensional numerical quadrature can be completely avoided, which means that the total error of the method is generated by three sources: (1) the discretization of the wave packet transform, (2) the thawed Gaussian approximations and (3) the numerical integration of the thawed equations of motion.
The precise analysis of these errors is the subject of this paper.\\

As expected, our analysis confirms that also for the TSGT method the conventional grid-based approach results in an unacceptably large number of basis functions since the total number of grid points increases exponentially with the dimension $d$ for achieving a given accuracy.
One way to bypass the curse of dimensionality for the resulting tensors of basis functions and coefficients is to use low-rank approximation techniques.
In our future research, we will explore the combination of the TSTG method with tensor-train (TT) approximations as introduced by Oseledets and Tyrtyshnikov, see \cite{Oseledets:2009,Oseledets:2011}.

\subsection{Main results and outline}
The paper is organized as follows.
In section~\ref{sec:TSTG Propagation} we review the TSTG method and provide a detailed mathematical formulation of all subroutines.
This includes the definition of quadrature-based analysis, synthesis and re-initialization operators, which are used later to investigate the discretization of the wave packet transform and allow for direct comparison with other methods that can also be used to solve the time-dependent Schr\"odinger equation.
To the best of our knowledge, this is the first time that a rigorous mathematical formulation of the TSTG method is presented.
Afterwards, we investigate the errors produced by the individual subroutines and their concatenation.
In section~\ref{sec:Discretizing the Wave Packet Transform} we analyze the error for the discretization of the wave packet transform, whereas section~\ref{sec:Methods for Propagating Gaussian Wave Packets} deals with thawed Gaussian approximations followed by an analysis of time discretization for both variationally and non-variationally evolving basis functions.
Our main new result Theorem~\ref{fact:TSTG_global}, the first rigorous error bound for the TSTG method, is presented in section~\ref{sec:Error Estimates for the Concatenation}.
Finally, the one-dimensional numerical experiments in section~\ref{sec:Numerical Results} support our theoretical results and illustrate the applicability of the TSTG method for simulations of quantum dynamics, including tunneling dynamics in a double-well potential.

\section{The TSTG Propagation Method}\label{sec:TSTG Propagation}
In this section we present a detailed description of the TSTG method, which is accomplished by deriving a rigorous mathematical formulation of all subroutines.
We introduce the analysis, synthesis and re-initialization operators and compare the method with other existing approaches.\\

Recall the definition of the wave packet $g_z$ in \eqref{eq:def_gz}.
For a complex symmetric matrix $C\in\C^{d\times d}$ with positive definite imaginary part (the set of all matrices with this property is known as the Siegel upper half-space, see \cite{Siegel:1939}, and is denoted by $\mathfrak{S}^+(d)$ in this paper) and all $x\in\R^d$, we set
\begin{align*}
	g(x)
	\,=g^C(x)
	:=\pi^{-d/4}\det(\operatorname{Im}C)^{1/4}\exp\left(\frac{i}{2}x^TC x\right),
\end{align*}
from which we obtain
\begin{equation}\label{def:gz}
	\begin{split}
		g_z(x)
		=g_z^{C,\e}(x)
		&=(\pi\e)^{-d/4}\det(\operatorname{Im}C)^{1/4}\cdots\\
		&\qquad\exp\left[\frac{i}{\e}\left(\frac{1}{2}(x-q)^TC(x-q)+p^T(x-q)\right)\right].
	\end{split}
\end{equation}
The dependency on $C$ and $\e$ is always assumed implicitly in the short-hand notation.\\

Based on the time-independent linear approximation space
\begin{align*}
	\mathcal{V}_{\mathcal{K}}
	=\operatorname{span}\big\{g_{\mathbf{k},0}\in L^2(\R^d)\,:\,\mathbf{k}\in\mathcal{K}\big\}
	\subset L^2(\R^d),
\end{align*} 
the TSTG method approximates the solution $\psi$ to the Schr\"odinger equation \eqref{eq:tdse} with time-dependent coefficients as follows:
\begin{align}\label{eq:ansatz}
	\psi(t)
	\approx\psi_{\mathcal{K}}(t)
	:=\sum_{\mathbf{k}\in\mathcal{K}}c_{\mathbf{k}}(t)\,g_{\mathbf{k},0}.
\end{align}
The time-dependent representation coefficients result from the concatenation of thawed Gaussian propagation steps for the basis functions with the re-initialization of the evolved basis in the time-independent approximation space $\mathcal{V}_{\mathcal{K}}$.
To give the equations of motion for the coefficients, we introduce the quadrature-based pair of operators
\begin{align*}
	\mathcal{A}_{\mathcal{K}}
	&\colon L^2(\R^d)\to\C^{\mathcal{K}},\,\psi\mapsto (c_{\mathbf{k}}(\psi)),\\
		\mathcal{S}_{\mathcal{K}}
	&\colon\C^{\mathcal{K}}\to\mathcal{V}_{\mathcal{K}},\,(c_{\mathbf{k}})\mapsto\psi_{\mathcal{K}}:=\sum_{\mathbf{k}\in\mathcal{K}}c_{\mathbf{k}}\,g_{\mathbf{k},0},
\end{align*}
where for a given quadrature formula the analysis operator $\mathcal{A}_{\mathcal{K}}$ maps a function $\psi\in L^2(\R^d)$ to the coefficient tensor $(c_{\mathbf{k}}(\psi))$ according to the discretization of the wave packet transform \eqref{eq:inverse_FBI_discrete} and the synthesis operator $\mathcal{S}_{\mathcal{K}}$ maps a given coefficient tensor $(c_{\mathbf{k}})$ to the Gaussian superposition
\begin{align*}
	\sum_{\mathbf{k}\in\mathcal{K}}c_{\mathbf{k}}\,g_{\mathbf{k},0}\in\mathcal{V}_{\mathcal{K}}.
\end{align*}
Furthermore, for a tensor $\mathcal{C}\in\C^{\mathcal{K}\times\mathcal{K}}$ we introduce the so-called re-initialization operator
\begin{align*}
	\mathcal{R}_{\mathcal{K}}(\mathcal{C})\colon\C^{\mathcal{K}}\to\C^{\mathcal{K}},\,
	(c_{\mathbf{k}})\mapsto\sum_{\mathbf{k}'\in\mathcal{K}}\mathcal{C}_{\mathbf{k},\mathbf{k}'}c_{\mathbf{k}'},
\end{align*}
which can be viewed as a multidimensional version of the matrix-vector product. 
With these operators at hand we can formulate the TSTG method, which starts to run through the following three subroutines once:
\begin{enumerate}
	\item[(s1)] \emph{Representation coefficients of the initial wave function:}\\[1mm]
		The first subroutine computes the coefficients 
		\begin{align*}
			(c_{\mathbf{k}}(\psi_0))
			=\mathcal{A}_{\mathcal{K}}\psi_0,
		\end{align*}
		which can be used to build the following approximation of a given initial wave function $\psi_0$ in $\mathcal{V}_{\mathcal{K}}$:
		\begin{align*}
			\psi_0
			\approx\mathcal{S}_{\mathcal{K}}\mathcal{A}_{\mathcal{K}}\psi_0
			=\sum_{\mathbf{k}\in\mathcal{K}}c_{\mathbf{k}}(\psi_0)\,g_{\mathbf{k},0}.
		\end{align*}
	\item[(s2)] \emph{Thawed Gaussian propagation of the basis functions:}\\[1mm]
		In the second subroutine, each basis function $g_{\mathbf{k},0}$ is propagated for a short propagation period $\tau>0$.
		More precisely, each individual time-evolved basis function $g_{\mathbf{k}}(\tau)$ is approximated by an element $u_{\mathbf{k}}(\tau)$ in the manifold of complex Gaussian functions
		\begin{align*}
			\mathcal{M}
			=\Bigg\{u\in L^2(\R^d)\,\Big|\,
			u(x)
			=g_z^{C,\e}(x)e^{iS/\e},\,z\in\R^{2d},\,C\in\mathfrak{S}^+(d),\,S\in\R\Bigg\},
		\end{align*}
		evolving according to the thawed Gaussian propagation method, see \cite{Heller:1975}.
		It is known that $u_{\mathbf{k}}(\tau)$ is an accurate approximation only if the potential can be approximated as harmonic throughout the ``support'' of $u_{\mathbf{k}}(\tau)$, i.e., as long as its width is not too wide, see Lemma~\ref{fact:thawed} for a precise estimate.
		Based on a numerical integrator for the corresponding equations of motion, let us introduce the approximate propagator
		\begin{align}\label{eq:integrator}
			\mathcal{U}_{\mathbf{k}}^\tau\colon\mathcal{M}\to\mathcal{M},\,
			g_{\mathbf{k},0}\mapsto u^\tau_{\mathbf{k}},
		\end{align}	
		where we use the notation with the superscript to indicate that $u^\tau_{\mathbf{k}}\in\mathcal{M}$ is the numerical approximation to $u_{\mathbf{k}}(\tau)$ obtained by solving a system of ordinary differential equations (see also section~\ref{sec:Methods for Propagating Gaussian Wave Packets}).
		Then, for all $\mathbf{k}\in\mathcal{K}$, the second subroutine produces the numerical approximants
		\begin{align*}
			u^\tau_{\mathbf{k}}
			=\mathcal{U}_{\mathbf{k}}^\tau\,g_{\mathbf{k},0}
			\approx g_{\mathbf{k}}(\tau).
		\end{align*}
	\item[(s3)] \emph{Computation of coefficients for the reinitialization:}\\[1mm]
		The approximants $u^\tau_{\mathbf{k}}$ obtained in (s2) are now re-expanded in $\mathcal{V}_{\mathcal{K}}$ as follows.
		For all $\mathbf{k}\in\mathcal{K}$, we apply the analysis operator $\mathcal{A}_{\mathcal{K}}$ to the wave packet $u^\tau_{\mathbf{k}}$, which gives us the tensors
		\begin{align*}
			\mathcal{A}_{\mathcal{K}}u^\tau_{\mathbf{k}}
			=(c_{\mathbf{k}'}(u^\tau_{\mathbf{k}}))\in\C^{\mathcal{K}}.
		\end{align*}
		The result of the third subroutine is then a tensor $\mathcal{C}^\tau\in\C^{\mathcal{K}\times\mathcal{K}}$ that contains the coefficients $\mathcal{C}^\tau_{\mathbf{k}',\mathbf{k}}:=c_{\mathbf{k}'}(u^\tau_{\mathbf{k}})$ for all $\mathbf{k},\mathbf{k}'\in\mathcal{K}$.
		In particular, this tensor is obtained without numerical integration, because all coefficients sample inner products of two Gaussians.
		\begin{remark}
			Note that the corresponding re-expansion $u^\tau_{\mathbf{k},\mathcal{K}}$ of $u^\tau_{\mathbf{k}}$ in $\mathcal{V}_{\mathcal{K}}$ is given by the action of the synthesis operator:
			\begin{align*}
				u^\tau_{\mathbf{k},\mathcal{K}}
				:=\mathcal{S}_{\mathcal{K}}\mathcal{A}_{\mathcal{K}}u^\tau_{\mathbf{k}}
				=\sum_{\mathbf{k}'\in\mathcal{K}}c_{\mathbf{k}'}(u^\tau_{\mathbf{k}})\,g_{\mathbf{k}',0}.
			\end{align*}
		\end{remark}
\end{enumerate}
Running through the above subroutines once, we are equipped with the tensor $(c_{\mathbf{k}}(\psi_0))$ for the approximation of the initial datum and the tensor $\mathcal{C}^\tau\in\C^{\mathcal{K}\times\mathcal{K}}$ containing the coefficients $c_{\mathbf{k}'}(u^\tau_{\mathbf{k}})$.
To now obtain an approximation of the solution at time $\tau$, we use the re-initialization operator $\mathcal{R}_{\mathcal{K}}^\tau:=\mathcal{R}_{\mathcal{K}}(\mathcal{C}^\tau)$ to get
\begin{equation*}
	\begin{split}
		\psi(\tau)
		&\overset{(s1)}{\approx} U(\tau)\mathcal{S}_{\mathcal{K}}\mathcal{A}_{\mathcal{K}}\psi_0
		\overset{(s2)}{\approx}\sum_{\mathbf{k}\in\mathcal{K}}c_{\mathbf{k}}(\psi_0)\,\mathcal{U}_{\mathbf{k}}^\tau\,g_{\mathbf{k},0}
		=\sum_{\mathbf{k}\in\mathcal{K}}c_{\mathbf{k}}(\psi_0)\,u^\tau_{\mathbf{k}}\\
		&\overset{(s3)}{\approx}\sum_{\mathbf{k}\in\mathcal{K}}c_{\mathbf{k}}(\psi_0)\,\mathcal{S}_{\mathcal{K}}\mathcal{A}_{\mathcal{K}}u^\tau_{\mathbf{k}}
		=\sum_{\mathbf{k}\in\mathcal{K}}c_{\mathbf{k}}(\psi_0)u^\tau_{\mathbf{k},\mathcal{K}}\\
		&\,\,=\,\,\sum_{\mathbf{k}\in\mathcal{K}}\left(\sum_{\mathbf{k}'\in\mathcal{K}}\mathcal{C}^\tau_{\mathbf{k},\mathbf{k}'}c_{\mathbf{k}'}(\psi_0)\right)g_{\mathbf{k},0}
		=\mathcal{S}_{\mathcal{K}}\mathcal{R}_{\mathcal{K}}^\tau\mathcal{A}_{\mathcal{K}}\psi_0
		=:\psi_{\mathcal{K}}^{1,\tau},
	\end{split}
\end{equation*}
where we have changed the names of the indices $\mathbf{k}$ and $\mathbf{k}'$ to get to the third line.
Furthermore, using that the unitary propagator can be decomposed for $n>1$ as
\begin{align*}
	U(n\tau)
	=U(\tau)\circ\cdots\circ U(\tau),
\end{align*}
single TSTG propagation steps can be concatenated to approximate the solution at times $2\tau,3\tau,\dots$, where for the $(n+1)$-th iteration we use the approximant $\psi_{\mathcal{K}}^{n,\tau}$ of the $n$-th iteration as new initial datum and therefore we arrive at the following approximation at time $t_n=n\tau$,
\begin{align*}
	\psi(t_n)
	\approx\mathcal{S}_{\mathcal{K}}\left(\mathcal{R}_{\mathcal{K}}^\tau\right)^n\mathcal{A}_{\mathcal{K}}\psi_0
	=:\psi_{\mathcal{K}}^{n,\tau},
\end{align*}
where we replaced the operator $\mathcal{A}_{\mathcal{K}}\mathcal{S}_{\mathcal{K}}$ in the intermediate steps with the identity, which reflects the fact that the representation coefficients from a previous step can be kept in memory.
In particular, the re-initialization yields that the corresponding coefficients of $\psi_{\mathcal{K}}^{n,\tau}$ are given for all $\mathbf{k}\in\mathcal{K}$ by the recursion formula
\begin{equation}\label{eq:cn}
	\begin{split}
		c^{n,\tau}_{\mathbf{k}}
		:=\left(\mathcal{R}_{\mathcal{K}}^\tau\right)^n\mathcal{A}_{\mathcal{K}}\psi_0
		&=\mathcal{R}_{\mathcal{K}}^\tau\left((\mathcal{R}_{\mathcal{K}}^\tau)^{n-1}\mathcal{A}_{\mathcal{K}}\psi_0\right)\\
		&=\sum_{\mathbf{k}'\in\mathcal{K}}c^{n-1,\tau}_{\mathbf{k'}}c_{\mathbf{k}}(u^\tau_{\mathbf{k}'}),
		\quad
		c^{0,\tau}_{\mathbf{k}}
		:=c_{\mathbf{k}}(\psi_0).
	\end{split}
\end{equation}
Finally, let us emphasize that the coefficients (and thus also the approximants) are updated recursively on the discrete time grid $2\tau,3\tau,\dots$ and therefore \eqref{eq:ansatz} should be rewritten for a fixed propagation time $\tau$ as
\begin{align*}
	\psi(t_n)
	\approx\psi_{\mathcal{K}}^{n,\tau}
	=\sum_{\mathbf{k}\in\mathcal{K}}c^{n,\tau}_{\mathbf{k}}\,g_{\mathbf{k},0}.
\end{align*}
\begin{remark}
	The TSTG method as originally introduced by Kong~\emph{et al.} does not use a direct discretization of the wave packet transform.
	Instead, the authors present an equivalent approach using a basis of closely overlapping Gaussians to construct a partition of unity based on a summation curve that can be approximated by a constant in the support of all basis functions.
	We examined this approach in \cite{Bergold:2020u} and the discretization of the wave packet transform presented here gives a new perspective that enables a straightforward representation of the discretization error.
\end{remark}
%

\subsection{Comparison with other methods}
Looking at the chosen ansatz in \eqref{eq:ansatz}, one way to determine the corresponding time-dependent coefficient tensor $c=(c_{\mathbf{k}})$ would be the standard Galerkin method, which yields a linear system of ordinary differential equations and is derived from the condition that
\begin{equation}\label{eq:Galerkin}
	\begin{split}
		&\partial_t\psi_{\mathcal{K}}(t)
		\in\mathcal{V}_{\mathcal{K}}
		\quad\text{is such that}\\
		&\big\langle\varphi\mid-i\e\partial_t\psi_{\mathcal{K}}(t)+H\psi_{\mathcal{K}}(t)\big\rangle
		=0
		\quad\text{for all $\varphi\in\mathcal{V}_{\mathcal{K}}$}.
	\end{split}
\end{equation}
With the orthogonal projection $P_{\mathcal{K}}\colon L^2(\R^d)\to\mathcal{V}_{\mathcal{K}}$ onto the approximation space, the Galerkin condition \eqref{eq:Galerkin} can also be written as
\begin{align*}
	i\e\partial_t\psi_{\mathcal{K}}
	=P_{\mathcal{K}}H\psi_{\mathcal{K}}.
\end{align*}
Let us therefore take a closer look at $\mathcal{V}_{\mathcal{K}}$.
The approximation space is spanned by the non-orthogonal Gaussian basis functions $g_{\mathbf{k},0}$.
To achieve a given accuracy for the discretization of the wave packet transform, the grid points $z_{\mathbf{k}}$ must be chosen sufficiently close, which means that the basis functions have a large overlap and therefore the Gram matrix of the Galerkin method becomes ill-conditioned.
This problem has been extensively studied in the literature, see e.g. \cite[section~3]{Fornberg:2015}, and several stabilization algorithms have been proposed, see e.g \cite{Fornberg:2011,Kormann:2019}.
Furthermore, it is worth noting that the Gram matrix becomes the identity if the Gaussians are replaced by an orthonormal basis and a comparison must be made with the Galerkin method in \cite[chapter~III.1.1]{Lubich:2008}, where the time-independent approximation space is spanned by the first $K\ge1$ Hermite functions
\begin{align*}
	\varphi_k(x)
	:=\frac{1}{\sqrt{2^kk!\sqrt{\pi}}}\frac{\mathrm{d}^k}{\mathrm{d}x^k}e^{-x^2},\quad
	k=0,1,\dots,K-1,\,x\in\R,
\end{align*}
which are known to form an $L^2$-orthonormal set.
Although this choice enables a convincingly simple representation of the orthogonal projection, namely
\begin{align*}
	P_{\mathcal{K}}
	=\sum_{k=0}^{K-1}\langle\varphi_k\mid\bullet\rangle\,\varphi_k,
\end{align*}
which is used in \cite[chapter~III.1.1, Theorem 1.2]{Lubich:2008} to derive the approximation error of the Galerkin method, in practical applications the dimension of $\mathcal{V}_\mathcal{K}$ must typically be chosen large in order to compute the evolution of the wave function with sufficient accuracy.
For instance, for simulations of tunneling in double-well potentials (quartic potentials with two local minima separated by energy barriers) as presented later in \S\ref{sub:One-dimensional double-well potential}, the Hermite basis is expensive since the Hermite functions are localized by a Gaussian envelope and therefore the degree of the polynomial prefactors must be large to capture both minima.

Furthermore, we note that time-varying approximation spaces have also been studied in the past.
Linear combinations of time-evolved frozen Gaussian functions have been proposed by Heller, see \cite{Heller:1981}, and can be improved by taking a linear combination of Dirac--Frenkel time-dependent coefficients, which are determined by the time-dependent variational principle, see \cite[chapter~II.5.3]{Lubich:2008}.
We would also like to mention the Galerkin approximation for Hagedorn functions, a generalization of the Hermite functions based on a Gaussian amplitude with arbitrary width matrix in the Siegel half space, see e.g. \cite[section~4.3]{Lasser:2020} and \cite{Gradinaru:2014,Blanes:2020}.

\subsection{Summary}
While the standard Galerkin condition yields a linear system of ordinary differential equations for the coefficients, which contains the ill-conditioned Gram matrix due to the closely overlapping basis functions, the TSTG method combines thawed Gaussians for the propagation of the basis with the operators $\mathcal{A}_{\mathcal{K}},\mathcal{S}_{\mathcal{K}}$ and $\mathcal{R}_{\mathcal{K}}^\tau$, which are based on the discretization of the wave packet transform and are obtained without numerical integration.

\section{Discretizing the Wave Packet Transform}\label{sec:Discretizing the Wave Packet Transform}
In this section we discuss the discretization of the phase space integral
\begin{align}\label{eq:integral_phase_space}
	(2\pi\e)^{-d}\int_{\R^{2d}}\left\langle g_z\mid\psi\right\rangle g_z\,\mathrm{d}z
	\approx\sum_{\mathbf{k}\in\mathcal{K}}c_{\mathbf{k}}(\psi)\,g_{\mathbf{k}}
\end{align}
for the case of Gaussian basis functions and uniform Riemann sums.
We present an analytical formula for the coefficients $c_{\mathbf{k}}(\psi)$, proving that they are Gaussian wave packets in phase space.
Moreover, we discuss the discretization error for \eqref{eq:integral_phase_space}.\\

Recall the inversion formula of the FBI transform in \eqref{eq:inverse_FBI}.
The first attempt to obtain an approximation of the phase space integral might use a multivariate integration formula based on weighted point evaluations of the integrand and for this case the analysis operator takes the form
\begin{align*}
	(\mathcal{A}_{\mathcal{K}}\psi)_{\mathbf{k}}
	=c_{\mathbf{k}}(\psi)
	=w_{\mathbf{k}}\langle g_{\mathbf{k}}\mid\psi\rangle,
	\quad\mathbf{k}\in\mathcal{K},
\end{align*}
where the numbers $w_{\mathbf{k}}\ge 0$ are non-negative weights.
In particular, in Appendix~\ref{sec:Analysis and Synthesis Operator} we prove that $\mathcal{A}_{\mathcal{K}}$ and $\mathcal{S}_{\mathcal{K}}$ are formally adjoint and therefore from now on we write $\mathcal{S}_{\mathcal{K}}=\mathcal{A}_{\mathcal{K}}^*$.
Since on the manifold $\mathcal{M}\subset L^2(\R^d)$ of complex Gaussian functions the analysis operator has an analytic representation, let us start to take a closer look at the inner products of Gaussians.
\begin{remark}
	The inversion formula of the FBI transform is known in the literature under different names, for instance as the inversion formula for the short-time Fourier transform in time-frequency analysis (the semiclassical parameter $\e$ is not considered in this context), see e.g. \cite[Corollary~3.2.3]{Grochenig:2001}, or, in presence of a Gaussian amplitude, as the inversion formula for the Gabor transform, see e.g. \cite[Eq.~3.2.5]{Feichtinger:1998}.
	Correspondingly, its discrete counterpart as considered here is related to Gabor frames.
	However, the coefficients as they result from a direct discretization of the phase space integral are not the exact Gabor coefficients and are obtained without computing the dual window of $g$.
	For a broader perspective on this theory we refer to \cite[chapter~5]{Grochenig:2001}.
\end{remark}
%

\subsection{Inner products of Gaussians}
The inner product of Gaussian wave packets has an explicit analytic expression and the next lemma shows that it can be written as a Gaussian in phase space.
\begin{lemma}\label{fact:inner}
	For $C_1,C_2\in\mathfrak{S}^+(d)$ in the Siegel space and $z_1,z_2\in\R^{2d}$ we have
	\begin{align}\label{eq:inner1}
		\langle g_{z_1}^{C_1,\e}\mid g_{z_2}^{C_2,\e}\rangle
		=\beta\exp
		\left(\frac{i}{2\e}(z_2-z_1)^TM(z_2-z_1)
		\right),
	\end{align}
	where the matrix
	\begin{align}\label{eq:inner2}
		M
		:=
		\begin{pmatrix}
			\left(C_2^{-1}-\bar C_1^{-1}\right)^{-1} & 0\\
			0 & -(C_2-\bar C_1)^{-1}
		\end{pmatrix}
		\in\C^{2d\times 2d}
	\end{align}
	is an element of the Siegel space $\mathfrak{S}^+(2d)$ of $2d\times 2d$ matrices and for $B=C_2-\bar C_1$ the complex constant $\beta\in\C$ is given by
	\begin{align*}
		\beta
		&:=\frac{2^{d/2}\det(\operatorname{Im}C_1\operatorname{Im}C_2)^{1/4}}{\sqrt{\det(-iB)}}\exp\left(\frac{i}{2\e}(p_1+p_2)^T(q_1-q_2)\right)\cdots\\
		&\qquad\exp\left(\frac{i}{2\e}(p_2-p_1)^TB^{-1}(C_2+\bar C_1)(q_2-q_1)\right).
	\end{align*}
	Moreover, if the eigenvalues of the positive definite matrices $\operatorname{Im}(C_k)$ and $\operatorname{Im}(-C_k^{-1})$, $k=1,2$, are bounded from below by a constant $\theta>0$, then the absolute value of the inner product is bounded by
	\begin{align}\label{eq:inner3}
		\left|\langle g_{z_1}^{C_1,\e}\mid g_{z_2}^{C_2,\e}\rangle\right|^2
		\le\zeta\exp\left(-\frac{\theta}{4\e}\|z_2-z_1\|_2^2\right),
	\end{align}
	where the constant $\zeta>0$ depends on $\theta$ and an upper bound on the eigenvalues of $\operatorname{Im}(C_k)$ and $\operatorname{Im}(-C_k^{-1})$, but is independent of $\e$.
\end{lemma}
We present the proof in Appendix~\ref{sec:Inner Products of Gaussians_appendix} and note that the bound in \eqref{eq:inner3} can easily be improved if the lower bound on the eigenvalues of $\operatorname{Im}(C_k)$ and $\operatorname{Im}(-C_k^{-1})$ is not chosen uniformly.
We also refer to the proof for the dependence of $\zeta$ on the spectral parameters.\\

From Lemma~\ref{fact:inner} we learn that the inner product $\langle g_{\mathbf{k}}\mid g_{z_0}^{C_0,\e}\rangle$, as it appears in \eqref{eq:integral_phase_space} for the choice $z=z_{\mathbf{k}}$ and $\psi=g_{z_0}^{C_0,\e}$, is a Gaussian in phase space:
\begin{lemma}\label{fact:gaussian_coefficients}
	For a Gaussian wave packet $\psi$, the coefficients $c_{\mathbf{k}}(\psi)$ that result from a discretization of the wave packet transform based on a multivariate quadrature formula are weighted Gaussian wave packets in phase space.
\end{lemma}
Due to the rapid decay of Gaussians, the (improper) phase space integral \eqref{eq:integral_phase_space} can be approximated by a truncated integral, which itself can be approximated via different multivariate quadrature rules afterwards.
In the next step we investigate these approximations.

\subsection{Truncation and multivariate quadrature}
We continue to investigate the truncation error for the wave packet transform.
\begin{lemma}[Truncation error]\label{fact:truncation}
	For a given phase space center $z_0\in\R^{2d}$ and a positive parameter $b>0$ consider the phase space box
	\begin{align}\label{eq:error_truncation1}
		B
		=\prod_{j=1}^{2d}[z_{0,j}-b,z_{0,j}+b]
		\subset\R^{2d}.
	\end{align}	
	Moreover, for $C,C_0\in\mathfrak{S}^+(d)$ let $g_z=g_z^{C,\e}$ and $\psi_0=g_{z_0}^{C_0,\e}$ and assume that the eigenvalues of $\operatorname{Im}(C),\operatorname{Im}(C_0)$ and $\operatorname{Im}(-C^{-1}),\operatorname{Im}(-C_0^{-1})$ are bounded from below by $\theta>0$ and from above by $\Theta>0$.
	Then, there exists a positive constant $c>0$, which is independent of $\e$ but depends on the spectral parameters, such that
	\begin{align}\label{eq:truncated_integral}
		\left\|\psi_0-(2\pi\e)^{-d}\int_B\langle g_z\mid\psi_0\rangle\,g_z\,\mathrm{d}z\right\|_{L^2(B_q)}
		\le c\exp\left(-\frac{d\theta}{4\e}b^2\right),
	\end{align}
	where $B_q\subset\R^d$ denotes the projection of $B$ onto the position space.
\end{lemma}
\begin{proof}
	Recall the definition of the Gaussian wave wave packet $g_z=g_z^{C,\e}$ in \eqref{def:gz}.
	A short calculation shows that in terms of the rescaled phase space box
	\begin{align*}
		B^\e
		:=\prod_{j=1}^{2d}[z_{0,j}^\e-b^\e,z_{0,j}^\e+b^\e],\quad
		z_0^\e
		:=z_0/\sqrt{\e},\quad 
		b^\e
		:=b/\sqrt{\e},
	\end{align*}
	the difference
	\begin{align*}
		f
		:=\psi_0-(2\pi\e)^{-d}\int_B\langle g_z\mid\psi_0\rangle\,g_z\,\mathrm{d}z
		=(2\pi\e)^{-d}\int_{\R^{2d}\setminus B}\langle g_z\mid g_{z_0}\rangle\,g_z\,\mathrm{d}z
	\end{align*}
	satisfies the following equation for all $x\in\R^d$:
	\begin{align*}
		\e^{d/4}f(\sqrt{\e}x)
		=(2\pi)^{-d}\int_{\R^{2d}\setminus B^\e}\left\langle g_{z'}^{C,1}\mid g_{z_0^\e}^{C,1}\right\rangle\,g_{z'}^{C,1}(x)\,\mathrm{d}z',
	\end{align*}
	which depends on $\e$ only through the semiclassically scaled truncation box $B^\e$.
	Since the scaling $f\mapsto\e^{d/4}f(\sqrt{\e}\bullet)$ is unitary and the Gaussian envelope $|g_{z'}^{C,1}|=|g^C(\bullet-q')|$ has unit $L^2$-norm, it further follows that
	\begin{align*}
		\|f\|^2_{L^2(B_q)}
		&=\int_{B_q^\e}\left|\e^{d/4}f(\sqrt{\e}x)\right|^2\mathrm{d}x\\
		&=\int_{B_q^\e}\left|(2\pi)^{-d}\int_{\R^{2d}\setminus B^\e}\left\langle g_{z'}^{C,1}\mid g_{z_0^\e}^{C,1}\right\rangle\,g_{z'}^{C,1}(x)\,\mathrm{d}z'\right|^2\mathrm{d}x\\
		&\le(2\pi)^{-2d}\int_{B_q^\e}\sup_{z'\in\R^{2d}\setminus B^\e}\left|g_{z'}^{C,1}(x)\right|^2\left(\int_{\R^{2d}\setminus B^\e}\left|\left\langle g_{z'}^{C,1}\mid g_{z_0^\e}^{C,1}\right\rangle\right|\,\mathrm{d}z'\right)^2\mathrm{d}x\\
		&\le(2\pi)^{-2d}\sup_{q'\in\R^d\setminus B_q^\e}\int_{\R^d}\left|g^C(x-q')\right|^2\mathrm{d}x\left(\int_{\R^{2d}\setminus B^\e}\left|\left\langle g_{z'}^{C,1}\mid g_{z_0^\e}^{C,1}\right\rangle\right|\,\mathrm{d}z'\right)^2\\
		&\le(2\pi)^{-2d}\left(\int_{\R^{2d}\setminus B^\e}\left|\left\langle g_{z'}^{C,1}\mid g_{z_0^\e}^{C,1}\right\rangle\right|\,\mathrm{d}z'\right)^2\mathrm{d}x
	\end{align*}
	and therefore the bound for the inner product of Gaussians in \eqref{eq:inner3} yields
	\begin{align*}
		\int_{B_q^\e}\left|\e^{d/4}f(\sqrt{\e}x)\right|^2\mathrm{d}x
		&\le(2\pi)^{-2d}\left(\frac{\Theta}{\theta}\right)^d\left(\int_{\R^{2d}\setminus B^\e}\exp\left(-\frac{\theta}{8}\|z'-z_0^\e\|^2_2\right)\mathrm{d}z'\right)^2.
	\end{align*}
	Furthermore, the symmetry of the integral and Fubini's theorem yields that
	\begin{align*}
		\int_{\R^{2d}\setminus B^\e}\exp\left(-\frac{\theta}{8}\|z'-z_0^\e\|_2^2\right)\,\mathrm{d}z'
		=\left(2\int_{b/\sqrt{\e}}^{\infty}\exp\left(-\frac{\theta}{8}y^2\right)\,\mathrm{d}y\right)^{2d}.
	\end{align*}
	Using the exponential-type bound $\operatorname{erfc}(z)\le e^{-z^2}$, $z>0$, for the complementary error function, see e.g. \cite[Eq.~(5)]{Chiani:2003}, we conclude that
	\begin{align*}
		\int_{b/\sqrt{\e}}^{\infty}\exp\left(-\frac{\theta}{8}y^2\right)\,\mathrm{d}y
		=\frac{\sqrt{2\pi}}{\sqrt{\theta}}\operatorname{erfc}\left(b\sqrt{\theta/8\e}\right)
		\le\frac{\sqrt{2\pi}}{\sqrt{\theta}}\exp\left(-\frac{\theta}{8\e}b^2\right),
	\end{align*}
	and therefore we finally get
	\begin{align*}
		\|f\|_{L^2(B_q)}
		&\le(2\pi)^{-d}\left(\frac{\Theta}{\theta}\right)^{d/2}\int_{\R^{2d}\setminus B^\e}\exp\left(-\frac{\theta}{8}\|z'-z_0^\e\|^2_2\right)\mathrm{d}z'\\
		&\le\left(\frac{\Theta}{\theta}\right)^{d/2}\left(\frac{2}{\sqrt{\theta}}\exp\left(-\frac{\theta}{8\e}b^2\right)\right)^{2d}
		=4^d\Theta^{d/2}\theta^{-3d/2}\exp\left(-\frac{d\theta}{4\e}b^2\right).
	\end{align*}
	In particular, this shows that the constant $c$ can be chosen as
	\begin{align*}
		c
		=4^d\Theta^{d/2}\theta^{-3d/2}.
	\end{align*}
\end{proof}
We note that Lemma~\ref{fact:truncation} can be easily improved if separate boxes $B_q\subset\R^d$ and $B_p\subset\R^d$ are used in position and momentum space, which can also be aligned with the eigenvectors of the width matrix of the integrand, see e.g. \cite[Lemma~3.4]{Bergold:2020u}.\\

The truncated phase space integral in \eqref{eq:truncated_integral} can now easily be approximated by a multidimensional Riemann sum over sufficiently dense lattices in position and momentum space.
This approach was used by Kong~\emph{et al.}, who worked with uniform grids of size $\Delta q_j>0$ and $\Delta p_j>0$ in each coordinate direction $j=1,\dots,d$, corresponding to constant weights
\begin{align*}
	w_{\mathbf{k}}
	=(2\pi\e)^{-d}\prod_{j=1}^d\Delta q_j\Delta p_j.
\end{align*}
For a given phase space box $B$ such as \eqref{eq:error_truncation1}, the discretization error then depends not only on the number of grid points that are used to subdivide $B$, but also on the dimension of the phase space:
\begin{lemma}\label{fact:multiRS}
	Let $f\in C^\infty(\R^{2d})$ and $K=k^{2d}$ for some $k\ge 1$.
	Then, there exists a positive constant $c_f>0$, depending on the function $f$, such that
	\begin{align*}
		\left|\int_{[0,1]^{2d}}f(z)\,\mathrm{d}z-\frac{1}{K}\sum_{\mathbf{k}\in\mathcal{K}}f\left(\frac{k_1}{k},\dots,\frac{k_{2d}}{k}\right)\right|
		\le c_f\cdot d\cdot k^{-1},
	\end{align*}
	where $\mathcal{K}=\{1,2,\dots,k\}^{2d}$.
	In particular, $c_f$ can be chosen as the total variation of the function $f$ (in the sense of Hardy and Krause).
\end{lemma}
We formulated Lemma~\ref{fact:multiRS} as a special variant of a more general result that can be found in \cite[chapter~5.5.5]{Davis:2007}.
Moreover, we note that the estimate in Lemma~\ref{fact:multiRS} can be improved to a bound of order $\mathcal{O}(k^{-2})$ if the composite midpoint rule is used instead of the composite rectangle rule.\\

The total error for the discretization of the wave packet transform is now obtained by combining the estimates in Lemma~\ref{fact:truncation} and Lemma~\ref{fact:multiRS}.
For a Gaussian wave packet $\psi_0=g_{z_0}^{C_0,\e}$ and a given phase space box $B$ centered in $z_0$, let $B_q\subset\R^d$ denote the projection of $B$ onto the position space.
Moreover, let us introduce the following notation for the spatial discretization error:
\begin{equation}\label{eq:error_wp}
	\begin{split}
		E_{wp}(\psi_0)
		=E_{wp}(\psi_0,B,\mathcal{K})
		:=
		\begin{cases}
			0, & \text{if $\psi_0\in\mathcal{V}_{\mathcal{K}}$}\\
			\|\psi_0-\mathcal{A}_{\mathcal{K}}^*\mathcal{A}_{\mathcal{K}}\psi_0\|_{L^2(B_q)}, & \text{else}.
		\end{cases}
	\end{split}
\end{equation}
Note that this definition reflects the assumption that on $\mathcal{V}_{\mathcal{K}}$ the operator $\mathcal{A}_{\mathcal{K}}^*\mathcal{A}_{\mathcal{K}}$ is replaced by the identity, since representation coefficients can be kept in memory.
We then arrive at the following error estimate:
\begin{proposition}[Discretization error for uniform Riemann sums]\label{fact:err_RS}\hfill\\
	Let $C_0\in\mathfrak{S}^+(d)$ and $z_0\in\R^{2d}$.
	For the discretization of the phase space integral
	\begin{align*}
		(2\pi\e)^{-d}\int_{\R^{2d}}\left\langle g_z\mid g_{z_0}^{C_0,\e}\right\rangle g_z\,\mathrm{d}z
	\end{align*}
	using the phase space box $B$ in \eqref{eq:error_truncation1} and uniform Riemann sums with $k\ge 1$ grid points in each coordinate direction, there exist constants $c^{(\operatorname{T)}},c^{(\operatorname{RS})}>0$ such that
	\begin{align*}
		E_{wp}
		\le c^{(\operatorname{T)}}+c^{(\operatorname{RS})}k^{-1}.
	\end{align*}
\end{proposition}
We learn from the previous discussions that the discretization of the phase space integral with conventional grid-based approaches such as Riemann sums in every coordinate direction results in an unacceptably large number of basis functions, since the total number of grid points increases exponentially with the dimension.
Sparse grid methods can overcome this curse of dimensionality to a certain extent, and we refer to \cite{Gerstner:1998} for a comprehensive presentation of several methods based on Smolyak's sparse grid construction and further developments.
As already mentioned, we plan to use tensor-train approximations to extend the dimensionality of dynamics simulable with the TSTG approach.
\begin{remark}
	In \cite{Bergold:2020u} we study the discretization of the wave packet transform via different quadrature rules.
	Based on Gauss--Hermite quadrature, we introduce a representation of Gaussian wave packets in which the number of basis functions is significantly reduced and therefore offers an alternative to the approximation with Riemann sums according to Proposition~\ref{fact:err_RS}. 
\end{remark}
%

\section{Methods for Propagating Gaussian Wave Packets}\label{sec:Methods for Propagating Gaussian Wave Packets}
This section deals with the propagation of the basis functions.
The main result is the error bound for a single TSTG step in Proposition~\ref{fact:TSTG1}, which combines an estimate for thawed Gaussian approximations with an estimate for the numerical integration of the underlying equations of motion.\\

Recall that in subroutine (s2) of the method the individual basis functions are propagated according to the (non-variational) thawed Gaussian equations, see \cite[Eq.~(17)]{Kong:2016}.
The equations for the parameters $z\in\R^{2d},\,C\in\mathfrak{S}^+(d)$ and $S\in\C$ in the definition of the manifold $\mathcal{M}$ combine the Hamiltonian system
\begin{align}\label{eq:hamilton}
	\dot z(t)
	=J\nabla h(z),\quad
	h(z)
	=\frac{1}{2}|p|^2+V(q),\quad
	J
	=
	\begin{pmatrix}
		0 & \operatorname{Id}_d\\
		-\operatorname{Id}_d & 0
	\end{pmatrix}
	\in\R^{2d\times 2d}
\end{align}
for the motion of the center $z(t)$ with equations for $C(t)$ and $S(t)$ ensuring that in the presence of a quadratic potential we obtain exact solutions.
In addition to the work done by Kong~\emph{et al.}, other propagation methods are also possible as long as the approximants $u_{\mathbf{k}}(\tau)$ lie in the Gaussian manifold $\mathcal{M}$ to ensure that the coefficients for the re-expansion can be calculated analytically.
For instance, variationally evolving Gaussians offer an alternative which, like the non-variational Gaussians, provide approximations with order $\mathcal{O}(\sqrt{\e})$ accuracy.
The approximate solution is then determined by the Dirac--Frenkel time-dependent variational approximation principle, see e.g. \cite[section~3]{Lasser:2020}, and the equations of motion for the parameters were first derived by Coalson and Karplus, see \cite{Coalson:1990}.
Using Hagedorn's parametrization $C=PQ^{-1}$, where the matrices $P,Q\in\C^{n\times n}$ are invertible and satisfy the relations
\begin{align}\label{eq:symplectic}
	Q^TP-P^TQ
	=0
	\quad\text{and}\quad
	Q^*P-P^*Q
	=2i\operatorname{Id},
\end{align}
these equations read
\begin{alignat}{3}\label{eq:eq_motion}
	\dot q
	&=p
	&&\qquad\text{and}\qquad
	\dot p
	&=&-\langle\nabla_x V\rangle_u,\notag\\[1mm]
	\dot Q
	&=P
	&&\qquad\text{and}\qquad
	\dot P\,
	&=&-\langle\nabla_x^2 V\rangle_uQ,
\end{alignat}
\vspace*{-5mm}
\begin{align*}
	S(t)
	=\int_0^t\left(\frac{1}{2}|p(s)|^2-\langle V\rangle_{u(s)}+\frac{\e}{4}\operatorname{tr}\Big(Q(s)^*\langle\nabla^2_x V\rangle_{u(s)}Q(s)\Big)\right)\mathrm{d}s,
\end{align*}
where we denote by $\langle W\rangle_u=\langle u\mid Wu\rangle,\,W\in\{V,\nabla_x V,\nabla^2_x V\}$, the expectation values.
In particular, for the propagation of the basis functions in the TSTG method the initial conditions are given by
\begin{align*}
	z_{\mathbf{k}}(0)
	=z_{\mathbf{k}},\,\,
	Q_{\mathbf{k}}(0)
	=\operatorname{Im}(C_0)^{-1/2},\,\,
	P_{\mathbf{k}}(0)
	=C_0Q_{\mathbf{k}}(0)
	\quad\text{and}\quad
	S_{\mathbf{k}}(0)
	=0,
\end{align*}
where $\operatorname{Im}(C_0)^{1/2}$ is the unique positive definite square root of $\operatorname{Im}(C_0)$.
\begin{remark}
	To get the equations of motion for the non-variational Gaussians as used by Kong~\emph{et al.}, we replace the equations in \eqref{eq:eq_motion} for $(q(t),p(t),Q(t),P(t))$ by the point evaluations
	\begin{alignat*}{3}
		\dot q
		&=p
		&&\qquad\text{and}\qquad
		\dot p
		&=&-V(q),\\[1mm]
		\dot Q
		&=P
		&&\qquad\text{and}\qquad
		\dot P\,
		&=&-\nabla_x^2 V(q)Q,
	\end{alignat*}
	\vspace*{-5mm}
	\begin{align*}
		S(t)
		=\int_0^t\left(\frac{1}{2}|p(s)|^2-V\big(q(s)\big)\right)\mathrm{d}s,
	\end{align*}
	which are computationally less demanding than the variational equations of motion.
	This implies that $Z(t)=(Q(t),P(t))$ is a solution to the linearization of the classical equations of motion,
	\begin{align*}
		\dot Z(t)
		=J\nabla^2h(z(t))Z(t),
	\end{align*}
	where the function $h$ and the symplectic matrix $J$ are defined according to \eqref{eq:hamilton}.
	Moreover, we note that in the presence of a quadratic potential the above equations coincide with those in \eqref{eq:eq_motion}.
	The parametrization in terms of $Z=(Q,P)$ goes back to the work of Hagedorn, see \cite{Hagedorn:1980,Hagedorn:1998} and the matrix conditions in \eqref{eq:symplectic} ensure the correct normalization of the approximant $u\in\mathcal{M}$.
\end{remark}
The next lemma presents the accuracy of the thawed Gaussian methods and extends the results for the $L^2$-error for variational Gaussians in \cite[Theorem~3.5]{Lasser:2020} to non-variational Gaussians.
We note that the first $L^2$-error for non-variational Gaussians was proved by Hagedorn, see \cite[Theorem~2.9]{Hagedorn:1998}.
\begin{lemma}\label{fact:thawed}
	Assume that
	\begin{itemize}
		\item the eigenvalues of the positive definite width matrix $\operatorname{Im}(C(t))$ are bounded from below by a constant $\rho>0$, for all $t\in[0,\tau]$.
		\item the potential function $V$ is three times continuously differentiable with a polynomially bounded third derivative.
	\end{itemize}
	Moreover, assume that $u(t)\in\mathcal{M}$ is an approximation to the Schr\"odinger equation that results from the thawed Gaussian method (variational or non-variational).
	Then, there exists a positive constant $c^{(1)}>0$ such that the error between the approximant $u(t)$ and the solution $\psi(t)$ is bounded in the $L^2$-norm by
	\begin{align}\label{eq:thawed1}
		\|u(t)-\psi(t)\|
		\le c^{(1)}\,t\sqrt{\e},
		\quad
		0\le t\le\tau,
	\end{align}
	where $c^{(1)}$ is independent of $\e$ and $t$ but depends on $\rho$.	
\end{lemma}
The crucial ingredient for the proof is the fact that both the variational and the non-variational approximation are exact, provided that the potential is quadratic, see \cite[Proposition~3.2]{Lasser:2020}, and therefore the estimate in \eqref{eq:thawed1} follows from a bound on the defect for the cubic part of the potential.
\begin{proof}
	Let $U_q\colon\R^d\to\R$ denote the second-order Taylor polynomial of $V$ at $q$ and let $W_q\colon\R^d\to\R$ be the corresponding remainder, i.e.,
	\begin{align*}
		V
		=U_q+W_q.
	\end{align*}
	Since the approximant $u(t)\in\mathcal{M}$ is the exact solution to
	\begin{align*}
		i\e\partial_t u(t)
		=-\frac{\e^2}{2}\Delta_x u(t)+U_{q(t)}u(t),
		\quad
		u(0)
		=\psi(0)
		=\psi_0,
	\end{align*}
	we obtain
	\begin{align*}
		\partial_t(u-\psi)
		=\frac{1}{i\e}H(u-\psi)-\frac{1}{i\e}W_qu,
	\end{align*}
	where
	\begin{align*}
		\|W_qu\|
		&=(\pi\e)^{-d/4}\det(\operatorname{Im}C)^{1/4}\cdots\\
		&\qquad\left(\int_{\R^d}|W_q(x)|^2\exp\left(-\frac{1}{\e}(x-q)^T\operatorname{Im}C(x-q)\right)\,\mathrm{d}x\right)^{1/2}.
	\end{align*}
	Moreover, using that $W_q(x)$ is the non-quadratic remainder at $q$, an estimate for moments of Gaussian functions (see \cite[Lemma~3.8]{Lasser:2020}) yields the existence of a constant $c^{(1)}>0$, depending on $\rho$, such that
	\begin{align*}
		\|W_qu\|
		\le c^{(1)}\,\e^{3/2}.
	\end{align*}
	Consequently, since $u-\psi$ satisfies the Schr\"odinger equation up to the defect
	\begin{align*}
		d(t)
		=-\frac{i}{\e}W_{q(t)}u(t),
	\end{align*}
	we finally conclude that
	\begin{align*}
		\|u(t)-\psi(t)\|
		\le\int_0^t\|d(s)\|\,\mathrm{d}s
		=\frac{1}{\e}\int_0^t\|W_{q(s)}u(s)\|\,\mathrm{d}s
		\le c^{(1)}t\sqrt{\e}.
	\end{align*}
\end{proof}
\begin{remark}
	We note that the equations of motion are different for the variational and the non-variational thawed Gaussian method and therefore we get individual lower bounds on the eigenvalues of the width matrix, so that, although we have omitted this dependency in our notation of Lemma~\ref{fact:thawed}, individual constants result for the two methods.
	In particular, the estimate of Lasser and Lubich for Gaussian moments show that $c^{(1)}$ depends on the third derivative of $V$ and is of order $\rho^{-3/2}$ with respect to the spectral parameter $\rho$.
	We also mention that, in contrast to the computation of the full wave function, the error in the expectation value of observables improves to an order $\mathcal{O}(\e)$ accuracy, see \cite[Theorem~3.5b]{Lasser:2020}.
\end{remark}
The estimate in \eqref{eq:thawed1} shows that the thawed Gaussian approximations produce errors that increase linearly in $t$, where a small semiclassical parameter yields an improvement by a factor $\sqrt{\e}$ for the corresponding constant.
Since we want to use the thawed Gaussians for the TSTG method to approximate the time-evolution of the basis functions $g_{\mathbf{k},0}$, we see that the propagation time $\tau$ must be chosen such that we get accurate approximations for all $\mathbf{k}\in\mathcal{K}$.
A good choice of $\tau$ therefore enables the control of the error for the propagation of the basis, but small values result in more concatenation steps in order to approximate the solution for a fixed final time (we present numerical experiments for the dependency on $\tau$ in \S\ref{sub:One-dimensional harmonic oscillator}).
With this in mind, let us note that frozen Gaussian approximations would also be possible, see \cite{Heller:1981}.
On the one hand, this leads to simpler equations of motion since these approximations do not need information about the second derivative of the potential, but on the other hand, with an eye on the parameter $\e$, the frozen Gaussian method reduce the order to $\mathcal{O}(1)$.\\

We now turn to the numerical integration for the equations of motion.

\subsection{Time discretization}
For the integration of the equations of motion we need a suitable numerical integrator.
In \eqref{eq:integrator} we therefore introduced the approximate propagator $\mathcal{U}_{\mathbf{k}}^\tau\colon\mathcal{M}\to\mathcal{M}$, which has not yet been defined in detail except that it maps a Gaussian basis function $g_{\mathbf{k},0}$ to some numerical approximation $u^\tau_{\mathbf{k}}\approx g_{\mathbf{k}}(\tau)$.
The development of such integrators essentially uses exponential operator splitting methods such as the first-order Lie splitting or the second-order Strang splitting, where we say that the integrator is of order $s\ge 1$, if there exists a constant $c^{(2)}>0$ such that the error between the approximant $u^\tau_{\mathbf{k}}$, obtained after $m\ge 1$ steps of size $h_\tau=\tau/m$, and the true solution $u_{\mathbf{k}}(\tau)$ is bounded in the $L^2$-norm by
\begin{align}\label{eq:thawed_integrator}
	\|u_{\mathbf{k}}^\tau-u_{\mathbf{k}}(\tau)\|_{L^2(\R^d)}
	\le c^{(2)}\tau\frac{h_\tau^s}{\e}.
\end{align}
For example, the $L^2$-error of Strang splitting is $\mathcal{O}(h_\tau^2/\e)$, which implies that the step size $h_\tau$ must be sufficiently smaller than $\sqrt{\e}$, and we refer to \cite{Descombes:2010} for rigorous error bounds in the semiclassical scaling $\e\ll 1$.\\

Equipped with a numerical integrator, we get the following error:
\begin{proposition}\label{fact:error_time_discrete}
	For $\tau>0$ and a uniform time grid of step size $h_\tau>0$ let 
	\begin{align}\label{eq:error_integrator1}
		E_{\mathbf{k}}^\tau
		=E_{\mathbf{k}}^\tau(h_\tau)
		:=\|u_{\mathbf{k}}^\tau-g_{\mathbf{k}}(\tau)\|_{L^2(\R)},
		\quad
		\mathbf{k}\in\mathcal{K}.
	\end{align}
	Moreover, assume that $\mathcal{U}_{\mathbf{k}}^\tau\colon\mathcal{M}\to\mathcal{M}$ is a numerical integrator of order $s\ge 1$.
	Then, under the hypotheses of Lemma~\ref{fact:thawed}, for all $\mathbf{k}\in\mathcal{K}$ there exists a positive constant $c_{\mathbf{k}}>0$ such that
	\begin{align}\label{eq:error_integrator2}
		E_{\mathbf{k}}^\tau
		\le c_{\mathbf{k}}\tau\left(\frac{h_\tau^s}{\e}+\sqrt{\e}\right).
	\end{align}
\end{proposition}
\begin{proof}
	Let $\tau>0$ and $m\ge 1$.
	For all $\mathbf{k}\in\mathcal{K}$, we combine the estimate in \eqref{eq:thawed1} with the estimate in \eqref{eq:thawed_integrator} to obtain
	\begin{align*}
		E_{\mathbf{k}}^\tau
		\le\|u_{\mathbf{k}}^\tau-u_{\mathbf{k}}(\tau)\|_{L^2(\R)}+\|u_{\mathbf{k}}(\tau)-g_{\mathbf{k}}(\tau)\|_{L^2(\R)}
		\le c_{\mathbf{k}}^{(2)}\tau\frac{h_\tau^s}{\e}+c_{\mathbf{k}}^{(1)}\tau\sqrt{\e}.
	\end{align*}
	Consequently, the bound in \eqref{eq:error_integrator2} follows for the constant
	\begin{align*}
		c_\mathbf{k}
		=\max_{\mathbf{k}\in\mathcal{K}}\left(c^{(1)}_{\mathbf{k}},c^{(2)}_{\mathbf{k}}\right).
	\end{align*}
\end{proof}
A practical second-order algorithm of the variational splitting was proposed and studied by Faou and Lubich, see \cite{Faou:2006}.
In particular, it conserves the norm and the symplecticity relations of the matrices $Q$ and $P$ in \eqref{eq:symplectic}.
Moreover, we note that there are various higher-order splittings for the unitary propagator that can also be used and refer the interested reader to \cite{McLachlan:2002} and \cite[chapter~III]{Hairer:2006}.\\

We are now equipped with an error estimate for the discretization of the wave packet transform, for the thawed Gaussian approximations and for the numerical integration of the thawed equations of motion.
We are therefore ready to analyze the error generated by a single TSTG step.
Afterwards, in Theorem~\ref{fact:TSTG_global} we lift this error estimate to a global one.

\subsection{Error after a single TSTG step}
Recall that a single TSTG step consists of the following approximations:
\begin{enumerate}
	\item the approximation of the initial wave function $\psi_0$ in the approximation space $\mathcal{V}_{\mathcal{K}}$ according to subroutine (s1)\\[-3mm]
	\item the propagation of the basis according to (s2)\\[-3mm]
	\item the re-expansion of the time-evolved basis in $\mathcal{V}_{\mathcal{K}}$ according to (s3)
\end{enumerate}
For a tensor $(c_{\mathbf{k}})\in\C^\mathcal{K}$ let us introduce the following notation for its 1-norm:
\begin{align*}
	\|c_{\mathbf{k}}\|_1
	:=\sum_{\mathbf{k}\in\mathcal{K}}|c_{\mathbf{k}}|.
\end{align*} 
We then obtain the following result:
\begin{proposition}[Error after a single TSTG step]\label{fact:TSTG1}
	For a given box $B\subset\R^{2d}$ in phase space and a finite index set $\mathcal{K}\subset\N^{2d}$ recall the definition of the spatial discretization error $E_{wp}$ in \eqref{eq:error_wp}.
	Moreover, for $\mathbf{k}\in\mathcal{K},\,\tau>0$ and $h_\tau>0$ recall the definition of the time discretization error $E_{\mathbf{k}}^\tau$ in \eqref{eq:error_integrator1} produced by a numerical propagator of order $s\ge 1$ for the thawed equations of motion.
	Then, there exists a positive constant $C>0$ such that
	\begin{align}\label{eq:TSTG1}
		\|\psi(\tau)-\sum_{\mathbf{k}\in\mathcal{K}}c^{1,\tau}_{\mathbf{k}}\,g_{\mathbf{k},0}\|_{L^2(B_q)}
		\le C\tau\left(\frac{h_\tau^s}{\e}+\sqrt{\e}\right)+E^{1,\tau},
	\end{align}
	where $E^{1,\tau}>0$ denotes the following bound for the total spatial discretization error:
	\begin{align*}
		E^{1,\tau}
		=E_{wp}(\psi_0)+C\tau\left(\frac{1+h_\tau^s}{\e}\right).
	\end{align*}
\end{proposition}
\begin{proof}
	In the following let $\|\bullet\|$ denote the $L^2$-norm on the box $B_q$ in position space.
	Using that the evolution operator $U(\tau)=e^{-iH\tau/\e}$ is unitary, we have
	\begin{align*}
		&\|\psi(\tau)-\sum_{\mathbf{k}\in\mathcal{K}}c^{1,\tau}_{\mathbf{k}}\,g_{\mathbf{k},0}\|
		=\|U(\tau)\psi_0-\sum_{\mathbf{k}\in\mathcal{K}}c^{1,\tau}_{\mathbf{k}}\,g_{\mathbf{k},0}\|\\
		&\qquad\le\|U(\tau)\left(\psi_0-\mathcal{A}_{\mathcal{K}}^*\mathcal{A}_{\mathcal{K}}\psi_0\right)+\mathcal{A}_{\mathcal{K}}^*\mathcal{A}_{\mathcal{K}}U(\tau)\psi_0-\sum_{\mathbf{k}\in\mathcal{K}}c^{1,\tau}_{\mathbf{k}}\,g_{\mathbf{k},0}\|\\
		&\qquad\le E_{wp}(\psi_0)+\|\sum_{\mathbf{k}\in\mathcal{K}}c_{\mathbf{k}}(\psi_0)\,g_{\mathbf{k}}(\tau)-\sum_{\mathbf{k}\in\mathcal{K}}c^{1,\tau}_{\mathbf{k}}\,g_{\mathbf{k},0}\|.
	\end{align*}
	For the second summand, the definition of the coefficients $c^{1,\tau}_{\mathbf{k}}$ in \eqref{eq:cn} yields
	\begin{equation}\label{eq:estimate_max}
		\begin{split}
			&\|\sum_{\mathbf{k}\in\mathcal{K}}c_{\mathbf{k}}(\psi_0)\,g_{\mathbf{k}}(\tau)-\sum_{\mathbf{k}\in\mathcal{K}}c^{1,\tau}_{\mathbf{k}}\,g_{\mathbf{k},0}\|\\
			&\qquad\le\sum_{\mathbf{k}\in\mathcal{K}}|c_{\mathbf{k}}(\psi_0)|\Big(\|\,g_{\mathbf{k}}(\tau)-u_{\mathbf{k}}^\tau\|+\|u_{\mathbf{k}}^\tau-\sum_{\mathbf{k}'\in\mathcal{K}}c_{\mathbf{k}'}(u^\tau_{\mathbf{k}})\,g_{\mathbf{k}',0}\|\Big)\\
			&\qquad\le\sum_{\mathbf{k}\in\mathcal{K}}|c_{\mathbf{k}}(\psi_0)|\Big(E_{\mathbf{k}}^\tau+E_{wp}(u^\tau_{\mathbf{k}})\Big).
		\end{split}
	\end{equation}
	In particular, as proved in Appendix~\ref{sec:Reconstruction Error}, for all $\mathbf{k}\in\mathcal{K}$ there exists a positive constant $\tilde c_{\mathbf{k}}>0$ such that
	\begin{align*}
		E_{wp}(u^\tau_{\mathbf{k}})
		\le\tilde c_{\mathbf{k}}\left(\frac{1+h_\tau^s}{\e}\right)\tau.
	\end{align*}
	Consequently, using the bound for $E_{\mathbf{k}}^\tau$ in \eqref{eq:error_integrator2} with the constant $c_{\mathbf{k}}>0$, the estimate in \eqref{eq:TSTG1} follows for the choice
	\begin{align*}
		C
		=\|c_{\mathbf{k}}(\psi_0)\|_1\cdot\max\left(\max_{\mathbf{k}\in\mathcal{K}}\,\tilde c_{\mathbf{k}},\,\max_{\mathbf{k}\in\mathcal{K}}\,c_{\mathbf{k}}\right).
	\end{align*}
\end{proof}
We note that \eqref{eq:TSTG1} combines the $1$-norm with the $\max$-norm to bound the last sum in \eqref{eq:estimate_max}.
Since the spatial errors $E_{wp}(u^\tau_{\mathbf{k}})$ will increase at the boundary of the grid $\{z_{\mathbf{k}}\}_{\mathbf{k}\in\mathcal{K}}$, but at the same time the coefficients $c_{\mathbf{k}}(\psi_0)$ decrease exponentially with the distance $\|z_{\mathbf{k}}-z_0\|_2$, other H\"older conjugate exponents, which reflect this grid-dependent interplay more accurately, could also be chosen.\\

In the next section we investigate the error that is produced by the concatenation of single TSTG steps.

\section{Error Estimates for the Concatenation}\label{sec:Error Estimates for the Concatenation}
As discussed in \S\ref{sec:TSTG Propagation}, approximations for larger times $2\tau,3\tau,\dots$ are based on the updated coefficients $c^{2,\tau}_{\mathbf{k}},c^{3,\tau}_{\mathbf{k}},\dots$ which are given by the recursion formula in \eqref{eq:cn}.
We therefore start to investigate the magnitude of these coefficients.\\

Recall that the time-evolved Gaussian approximants $u_{\mathbf{k}}^\tau\in\mathcal{M}$ are re-expanded in the original basis of Gaussians $g_{\mathbf{k},0}$, which gives us the updated coefficients
\begin{align*}
	c_{\mathbf{k}}^{1,\tau}
	=\mathcal{R}_{\mathcal{K}}^\tau\mathcal{A}_{\mathcal{K}}\psi_0
	=\sum_{\mathbf{k}'\in\mathcal{K}}c_{\mathbf{k}'}(\psi_0)c_{\mathbf{k}}(u^\tau_{\mathbf{k}'}).
\end{align*}
Since both factors $c_{\mathbf{k}'}(\psi_0)$ and $c_{\mathbf{k}}(u^\tau_{\mathbf{k}'})$ are Gaussian wave packets in phase space, the coefficients can be bounded by a Gaussian envelope (as a sum of Gaussians) and therefore, by induction on $n$, Gaussian bounds can be derived for all higher-order coefficients $c^{n,\tau}_{\mathbf{k}},n>1$:
\begin{proposition}\label{fact:exp_decay}
	For $z_0\in\R^{2d}$ and $C_0\in\mathfrak{S}^+(d)$ let $\psi_0=g_{z_0}^{C_0,\e}$ and $\{z_{\mathbf{k}}\}_{\mathbf{k}\in\mathcal{K}}$ be an arbitrary grid in phase space.
	Then, for all $n\ge 0$ and $\tau>0$, there exist positive constants $\zeta_n^{\e,\tau},\theta_n^\tau>0$ such that for all $\mathbf{k}\in\mathcal{K}$ we have
	\begin{align}\label{eq:exp_decay1}
		|c^{n,\tau}_{\mathbf{k}}|
		\le\zeta_n^{\e,\tau}\exp\left(-\frac{\theta_n^\tau}{8\e}\|z_{\mathbf{k}}-z_0\|_2^2\right).
	\end{align}
\end{proposition}
For the proof of Proposition~\ref{fact:exp_decay} we first derive an auxiliary result that allows us to bound the representation coefficients $c_{\mathbf{k}}(u^\tau_{\mathbf{k}'})$ of the time-evolved Gaussian approximant $u^\tau_{\mathbf{k}'}$, which according to Lemma~\ref{fact:gaussian_coefficients} is a Gaussian in phase space centered at $z_{\mathbf{k}'}(\tau)$, by a Gaussian envelope centered at $z_{\mathbf{k}'}=z_{\mathbf{k}'}(0)$.
\begin{lemma}\label{fact:shifted_envelope}
	Under the assumptions of Proposition~\ref{fact:exp_decay}, for all $\mathbf{k}'\in\mathcal{K}$, there exist positive constants $\zeta_{\mathbf{k}'}^\tau>0$ and $\theta_{\mathbf{k}'}^\tau>0$ such that for all $\mathbf{k}\in\mathcal{K}$ we have 	
	\begin{align}\label{eq:shifted_envelope1}
		|c_{\mathbf{k}}(u^\tau_{\mathbf{k}'})|
		\le\zeta_{\mathbf{k}'}^\tau\exp\left(-\frac{\theta_{\mathbf{k}'}^\tau}{8\e}\|z_{\mathbf{k}}-z_{\mathbf{k}'}\|^2\right).
	\end{align}
\end{lemma}
\begin{proof}
	Let $\mathbf{k}'\in\mathcal{K}$ and $\tau>0$.
	The definition of the coefficients $\mathcal{A}_{\mathcal{K}}u^\tau_{\mathbf{k}'}$ implies
	\begin{align*}
		|c_{\mathbf{k}}(u^\tau_{\mathbf{k}'})|
		=|(\mathcal{A}_{\mathcal{K}}u^\tau_{\mathbf{k}'})_{\mathbf{k}}|
		=w_{\mathbf{k}}|\langle g_{\mathbf{k}},u^\tau_{\mathbf{k}'}\rangle|
		\quad\text{for all $\mathbf{k}\in\mathcal{K}$},
	\end{align*}
	where the non-negative weights $w_{\mathbf{k}}\ge 0$ depend on the underlying quadrature rule and therefore, using Lemma~\ref{fact:inner}, we find constants $\beta_{\mathbf{k}'}^\tau,\theta_{\mathbf{k}'}^\tau>0$ such that 
	\begin{align*}
		|c_{\mathbf{k}}(u^\tau_{\mathbf{k}'})|
		\le\beta_{\mathbf{k}'}^\tau\exp\left(-\frac{\theta_{\mathbf{k}'}^\tau}{8\e}\|z_{\mathbf{k}}-z_{\mathbf{k}'}(\tau)\|_2^2\right),
	\end{align*}
	where $z_{\mathbf{k}'}(\tau)\in\R^{2d}$ is the center of the evolved basis function $u^\tau_{\mathbf{k}'}\in\mathcal{M}$.
	To bound this Gaussian envelope by a re-shifted envelope centered at the original point $z_{\mathbf{k}'}$ instead of the evolved center $z_{\mathbf{k}'}(\tau)$, we write the time-evolved grid in terms of the original grid as
	\begin{align*}
		z_{\mathbf{k}'}(\tau)
		=z_{\mathbf{k}'}+\delta_{\mathbf{k}'}(\tau)
	\end{align*}
	and introduce the maximal phase space shift
	\begin{align*}
		\delta(\tau)
		:=\max_{\mathbf{k}'\in\mathcal{K}}\|\delta_{\mathbf{k}'}(\tau)\|_2.
	\end{align*}
	Using the Cauchy--Schwarz inequality in $\R^d$, it then follows that
	\begin{align*}
		&\exp\left(-\frac{\theta_{\mathbf{k}'}^\tau}{8\e}\|z_{\mathbf{k}}-z_{\mathbf{k}'}(\tau)\|_2^2\right)
		=\exp\left(-\frac{\theta_{\mathbf{k}'}^\tau}{8\e}\|z_{\mathbf{k}}-z_{\mathbf{k}'}-\delta_{\mathbf{k}'}(\tau)\|_2^2\right)\\[1mm]
		&\qquad=\exp\left(-\frac{\theta_{\mathbf{k}'}^\tau}{8\e}\|z_{\mathbf{k}}-z_{\mathbf{k}'}\|_2^2\right)
		\exp\left(\frac{\theta_{\mathbf{k}'}^\tau}{4\e}\delta_{\mathbf{k}'}(\tau)^T(z_{\mathbf{k}}-z_{\mathbf{k}'})\right)
		\exp\left(-\frac{\theta_{\mathbf{k}'}^\tau}{8\e}\|\delta_{\mathbf{k}'}(\tau)\|_2^2\right)\\[1mm]
		&\qquad\le\exp\left(-\frac{\theta_{\mathbf{k}'}^\tau}{8\e}\|z_{\mathbf{k}}-z_{\mathbf{k}'}\|_2^2\right)
		\exp\left(\frac{\theta_{\mathbf{k}'}^\tau}{4\e}\delta(\tau)\|z_{\mathbf{k}}-z_{\mathbf{k}'}\|_2\right).
	\end{align*}
	Hence, if we denote by $D_{\max}>0$ the maximal distance $\|z_{\mathbf{k}}-z_{\mathbf{k}'}\|_2$ between two grid points in phase space and
	\begin{align*}
		\beta^\tau
		:=\max_{\mathbf{k}'\in\mathcal{K}}\,\exp\left(\frac{\theta_{\mathbf{k}'}^\tau}{4\e}\delta(\tau)D_{\max}\right),
	\end{align*}
	the bound in \eqref{eq:shifted_envelope1} follows for $\zeta_{\mathbf{k}'}^\tau=\beta_{\mathbf{k}'}^\tau\beta^\tau$.
\end{proof}
\begin{proof}[Proof (of Proposition~\ref{fact:exp_decay})]
	We present a proof by induction on $n\ge 0$.
	For $n=0$, the bound in \eqref{eq:exp_decay1} follows from Lemma~\ref{fact:shifted_envelope} if we replace $u^\tau_{\mathbf{k}'}$ by $\psi_0$.
	In particular, for this special case, the constants $\zeta_0^{\e,\tau}$ and $\theta_0^\tau$ do not depend on either $\e$ or $\tau$ and thus we could also write $\zeta_0$ and $\theta_0$.
	Now, let $n>1$ and assume that the bound in \eqref{eq:exp_decay1} holds for $n-1$.
	The recursion formula \eqref{eq:cn} yields
	\begin{align*}
		|c^{n,\tau}_{\mathbf{k}}|
		\le\sum_{\mathbf{k}'\in\mathcal{K}}|c^{n-1,\tau}_{\mathbf{k'}}||c_{\mathbf{k}}(u^\tau_{\mathbf{k}'})|
		\quad\text{for all $\mathbf{k}\in\mathcal{K}$},
	\end{align*}
	where the factor $|c^{n-1,\tau}_{\mathbf{k'}}|$ can be estimated according to the induction hypothesis and the second factor $|c_{\mathbf{k}}(u^\tau_{\mathbf{k}'})|$ according to Lemma~\ref{fact:shifted_envelope}.
	This means that we find constants $\zeta_{n-1}^{\e,\tau},\theta_{n-1}^\tau>0$ and $\zeta_{\mathbf{k}'}^\tau,\theta_{\mathbf{k}'}^\tau>0$ such that
	\begin{align*}
		|c^{n-1,\tau}_{\mathbf{k'}}|
		&\le\zeta_{n-1}^{\e,\tau}\exp\left(-\frac{\theta_{n-1}^\tau}{8\e}\|z_{\mathbf{k}'}-z_0\|_2^2\right)
		\quad\text{and}\\[1mm]
		|c_{\mathbf{k}}(u^\tau_{\mathbf{k}'})|
		&\le\zeta_{\mathbf{k}'}^\tau\exp\left(-\frac{\theta_{\mathbf{k}'}^\tau}{8\e}\|z_{\mathbf{k}}-z_{\mathbf{k}'}\|^2\right),
	\end{align*}
	and therefore we conclude that
	\begin{align*}
		\sum_{\mathbf{k}'\in\mathcal{K}}|c^{n-1,\tau}_{\mathbf{k'}}||c_{\mathbf{k}}(u^\tau_{\mathbf{k}'})|
		\le\zeta_{n-1}^{\e,\tau}\zeta^\tau\sum_{\mathbf{k}'\in\mathcal{K}}\exp\left(-\frac{\theta_{n-1}^\tau}{8\e}\|\tilde z_{\mathbf{k}'}\|_2^2\right)\exp\left(-\frac{\theta^\tau}{8\e}\|\tilde z_{\mathbf{k}}-\tilde z_{\mathbf{k}'}\|_2^2\right),
	\end{align*}
	where we introduced
	\begin{align*}
		\zeta^\tau
		:=\max_{\mathbf{k}'\in\mathcal{K}}\zeta_{\mathbf{k}'}^\tau
		>0
		\quad\text{and}\quad
		\theta^\tau
		:=\min_{\mathbf{k}'\in\mathcal{K}}\theta_{\mathbf{k}'}^\tau
		>0,
	\end{align*}
	as well as the shifted grid points $\tilde z_{\mathbf{k}}:=z_{\mathbf{k}}-z_0$.
	In Appendix~\ref{sec:Discrete Gaussian Convolution} we show that there exists a positive constant $c>0$, depending on $\theta_{n-1}^\tau,\theta^\tau,\e$ and the phase space grid, such that for all one-dimensional components $j=1,\dots,2d$ we have
	\begin{align*}
		&\sum_{k_j'\in\mathcal{K}_j}\exp\left(-\frac{\theta_{n-1}^\tau}{8\e}\left(z^{(j)}_{\mathbf{k}}\right)^2\right)\exp\left(-\frac{\theta^\tau}{8\e}\left(\tilde z^{(j)}_{\mathbf{k}}-\tilde z^{(j)}_{\mathbf{k}'}\right)^2\right)\\
		&\qquad\le c\exp\left(-\frac{1}{8\e}\frac{\theta_{n-1}^\tau\theta^\tau}{\theta_{n-1}^\tau+\theta^\tau}\left(\tilde z^{(j)}_{\mathbf{k}}\right)^2\right).
	\end{align*}
	Consequently, using the definition of the shifted grid $\tilde z_{\mathbf{k}}=z_{\mathbf{k}}-z_0$, we finally get	
	\begin{align*}
		&\sum_{\mathbf{k}'\in\mathcal{K}}\exp\left(-\frac{\theta_{n-1}^\tau}{8\e}\|\tilde z_{\mathbf{k}'}\|_2^2\right)\exp\left(-\frac{\theta^\tau}{8\e}\|\tilde z_{\mathbf{k}}-\tilde z_{\mathbf{k}'}\|_2^2\right)\\
		&\qquad\le c^{2d}\exp\left(-\frac{1}{8\e}\frac{\theta_{n-1}^\tau\theta^\tau}{\theta_{n-1}^\tau+\theta^\tau}\|z_{\mathbf{k}}-z_0\|_2^2\right),
	\end{align*}
	which proves the bound in \eqref{eq:exp_decay1} for
	\begin{align*}
		\zeta_n^{\e,\tau}
		=\zeta_{n-1}^{\e,\tau}\zeta^\tau c^{2d}
		\quad\text{and}\quad
		\theta_n^\tau
		=\frac{\theta_{n-1}^\tau\theta^\tau}{\theta_{n-1}^\tau+\theta^\tau}.
	\end{align*}
\end{proof}
The last proposition provides a bound for the magnitude of the coefficients $|c^{n,\tau}_{\mathbf{k}}|$.
Together with the error bound for a single TSTG step in Proposition~\ref{fact:TSTG1}, we are now ready to present the error bound for the concatenation.

\subsection{Global error estimate for the concatenation}
From Proposition~\ref{fact:TSTG1} we learn that the total error of a single TSTG propagation step can be decomposed into a time and a spatial component.
In particular, the error with respect to time consists of the error for the thawed Gaussian approximation of order $\mathcal{O}(\sqrt{\e})$ and the error for the numerical integration of order $\mathcal{O}(h_\tau^s/\e)$, whereas the spatial error consists of the error for the approximation of the initial datum $\psi_0$ in $\mathcal{V}_{\mathcal{K}}$ and the error for re-expansion of the time-evolved approximant $u_{\mathbf{k}}^\tau$ in $\mathcal{V}_{\mathcal{K}}$.
Our finial result generalizes this result for the concatenation of $n>1$ TSTG steps:
\begin{theorem}\label{fact:TSTG_global}
	Under the hypotheses of Proposition~\ref{fact:TSTG1}, there exists a positive constant $C>0$ such that the global error of the TSTG propagation method with $n\ge 1$ concatenated steps at time $t_n=n\tau$ is given by
	\begin{align}\label{eq:TSTG_global1}
		\|\psi(t_n)-\sum_{\mathbf{k}\in\mathcal{K}}c^{n,\tau}_{\mathbf{k}}\,g_{\mathbf{k},0}\|_{L^2(B_q)}
		\le Ct_n\left(\frac{h_\tau^s}{\e}+\sqrt{\e}\right)+E^{n,\tau},
	\end{align}
	where $E^{n,\tau}>0$ denotes the following bound for the total spatial discretization error:
	\begin{align*}
		E^{n,\tau}
		=E_{wp}(\psi_0)+Ct_n\left(\frac{1+h_\tau^s}{\e}\right).
	\end{align*}
\end{theorem}
\begin{remark}
	Recall that $h_\tau$ is the step size of the numerical integrator for the underlying system of ODEs in section~\ref{sec:Methods for Propagating Gaussian Wave Packets}, while $\tau$ is the TSTG step size.
	In particular, one typically chooses $h_\tau=\tau/m$ for a positive integer $m\ge 1$.
	Moreover, we note that, in order to balance the error in
	\begin{align*}
		\frac{h_\tau^s}{\e}+\sqrt{\e},
	\end{align*}
	one obtains the condition $h_\tau=\mathcal{O}(\e^{3/2s})$, where $s\ge 1$ is the order of the integrator.
	In particular, we get $h_\tau=\mathcal{O}(\e^{3/2})$ for $s=1$ and $h_\tau=\mathcal{O}(\e^{3/4})$ for $s=2$, which does not seem to be as efficient as the Gaussian beam method at $h_\tau=\mathcal{O}(\sqrt{\e})$.
	However, for the TSTG method, numerical integration only needs to be performed for the time interval $[0,\tau]$ since the numerical solution at time $t_n=n\tau$ is obtained by concatenating $n>1$ TSTG steps, without additional numerical integration but only via the computation of the update coefficients $c^{n,\tau}_{\mathbf{k}}$.
	Therefore, the total number $N_{\operatorname{GB}}=\mathcal{O}(t_n/\sqrt{\e})$ of time steps for the Gaussian beam method must be compared with $N_{\operatorname{TSTG}}=\mathcal{O}(\tau/\e^{-3/2s})$.
\end{remark}
\begin{proof}
	Again, let $\|\bullet\|$ denote the $L^2$-norm on $B_q$.
	For $n\ge 1$ we define
	\begin{align*}
		e_{n,\tau}
		:=\|\psi(n\tau)-\mathcal{A}_{\mathcal{K}}^*\left(\mathcal{R}_{\mathcal{K}}^\tau\right)^n\mathcal{A}_{\mathcal{K}}\psi_0\|
		=\|\psi(n\tau)-\psi^{n,\tau}\|.
	\end{align*}
	Using that $U(\tau)$ is unitary, we obtain the recursion
	\begin{align*}
		e_{n+1,\tau}
		&=\|U(\tau)\psi(n\tau)-\psi^{n+1,\tau}\|
		=\|U(\tau)\big(\psi(n\tau)-\psi^{n,\tau}+\psi^{n,\tau}\big)-\psi^{n+1,\tau}\|\\
		&\le\|\psi(n\tau)-\psi^{n,\tau}\|+\|U(\tau)\psi^{n,\tau}-\psi^{n+1,\tau}\|
		=e_{n,\tau}+\|U(\tau)\psi^{n,\tau}-\psi^{n+1,\tau}\|,
	\end{align*}
	where the second summand is the local error of the $n$-th step.
	Hence, the global error $e_{n,\tau}$ after $n$ steps can be expressed in terms of the local errors as
	\begin{align*}
		e_{n,\tau}
		=e_{1,\tau}+\sum_{l=1}^{n-1}\|U(\tau)\psi^{l,\tau}-\psi^{l+1,\tau}\|.
	\end{align*}
	We note that $e_{1,\tau}$ is the error after a single propagation step in Proposition~\ref{fact:TSTG1}.
	Furthermore, for $1\le l\le n-1$ the definition of the coefficients $c^{l,\tau}_{\mathbf{k}}$ in \eqref{eq:cn} yields
	\begin{equation*}
		\begin{split}
			&\|U(\tau)\psi^{l,\tau}-\psi^{l+1,\tau}\|
			=\|\sum_{\mathbf{k}\in\mathcal{K}}c^{l,\tau}_{\mathbf{k}}\,g_{\mathbf{k}}(\tau)-\psi^{l+1,\tau}\|\\
			&\qquad\le\sum_{\mathbf{k}\in\mathcal{K}}|c^{l,\tau}_{\mathbf{k}}|\Big(\|\,g_{\mathbf{k}}(\tau)-u_{\mathbf{k}}^\tau\|+\|u_{\mathbf{k}}^\tau-\sum_{\mathbf{k}'\in\mathcal{K}}c_{\mathbf{k}'}(u^\tau_{\mathbf{k}})\,g_{\mathbf{k}',0}\|\Big)\\
			&\qquad\le\sum_{\mathbf{k}\in\mathcal{K}}|c^{l,\tau}_{\mathbf{k}}|\Big(E_{\mathbf{k}}^\tau+E_{wp}(u^\tau_{\mathbf{k}})\Big).
		\end{split}
	\end{equation*}
	Consequently, using once more the bounds for $E_{\mathbf{k}}^\tau$ in \eqref{eq:error_integrator2} and for $E_{wp}(u^\tau_{\mathbf{k}})$ in Appendix~\ref{sec:Reconstruction Error} with corresponding constants $c_{\mathbf{k}}$ and $\tilde c_{\mathbf{k}}$, respectively, the bound in \eqref{eq:TSTG_global1} follows for the constant
	\begin{align*}
		C
		=c^{max}\max_{l=0,\dots,n-1}\|c^{l,\tau}_{\mathbf{k}}\|_1,
		\quad\text{where}\quad
		c^{max}
		:=\max\left(\max_{\mathbf{k}\in\mathcal{K}}\,\tilde c_{\mathbf{k}},\,\max_{\mathbf{k}\in\mathcal{K}}\,c_{\mathbf{k}}\right).
	\end{align*}
\end{proof}
The previous theorem proves that the error for the TSTG propagation increases linearly with the number $n$ of propagation steps, where the corresponding constant depends on the errors introduced by the discretization of the wave packet transform, the thawed Gaussian approximation and the integration of the equations of motion.
For the numerical experiments presented in the next section, we examine an error bound based on a direct computation of
\begin{align}\label{eq:error}
	\operatorname{err}_{\mathcal{K}}^{l,\tau}
	:=\sum_{\mathbf{k}\in\mathcal{K}}|c^{l,\tau}_{\mathbf{k}}|\Big(E_{\mathbf{k}}^\tau+E_{wp}(u^\tau_{\mathbf{k}})\Big)
\end{align}
for all $l=1,2,\dots,n-1$, using the split-step Fourier method for the propagation of the basis functions. Future research will address the derivation of a practical \emph{a posteriori} error bound to be used in \eqref{eq:error} for implementing the TSTG method with adaptive step sizes or adaptive mesh refinements.

\section{Numerical Results}\label{sec:Numerical Results}
We demonstrate the capabilities of the TSTG method with a series of examples.
We first examine the discretization of the wave packet transform that is used to decompose the initial wave function and for the re-expansion of the time-evolved basis as described in \S\ref{sec:Discretizing the Wave Packet Transform}.
Afterwards, we test the method by computing the full wave function of the one-dimensional harmonic oscillator for different propagation times~$\tau$ and step sizes $h_\tau$.
Moreover, we reproduce the numerical results of Kong~\emph{et al.} for a one-dimensional double-well potential.
In addition to Kong~\emph{et al.}, who used non-variationally evolving Gaussians for the propagation of the basis functions, we also used variational Gaussians to compare both methods.
\begin{remark}
	The following numerical examples support the main result presented in Theorem~\ref{fact:TSTG_global} and show that the estimate in \eqref{eq:TSTG_global1} is indeed a workable error bound.
	Our experiments show how the errors depend on the underlying method for propagating the basis functions (variationally vs. non-variationally evolving thawed Gaussians).
	Since the capabilities of the TSTG method itself have already been presented by Kong~\emph{et al.}, we concentrate on one-dimensional numerical experiments for the error analysis.
	For multidimensional numerical experiments on the TSTG method and comparison with other methods, we refer to \cite[Results]{Kong:2016}.
\end{remark}
%

\subsection{Approximation of the initial wave function}\hfill\\[2mm]
We present numerical experiments for the approximation of a Gaussian wave function with uniform Riemann sums according to Proposition~\ref{fact:err_RS} for
\begin{align}\label{eq:initial_dw}
	\psi_0(x)
	=(\pi\e)^{-1/4}\exp\left(-\frac{1}{2\e}(x+\sqrt{2\eta})^2\right),\quad
	\eta
	=1.3544,
\end{align}
which is later used in \S\ref{sub:One-dimensional double-well potential} as initial wave function for the double-well potential.
Figure~\ref{fig1} shows the reconstruction errors in the supremum norm as a function of grid points for different truncation boxes $B=[-b_q,b_q]\times[-b_p,b_p]$, where we used the same number of grid points for both intervals.
\begin{figure}
	\includegraphics[width=\textwidth]{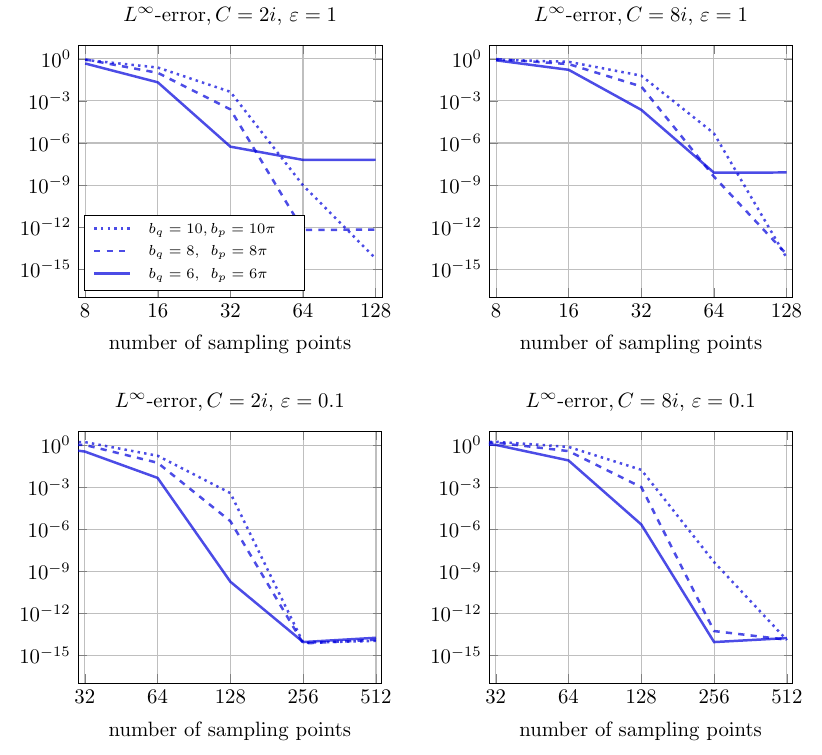}
	\caption{The initial wave function is approximated using the discretization of the wave packet transform.
		The four panels show the approximation errors for different combinations of $C$ (width of the basis functions) and $\e$.
		The number of sampling points that are used to achieve a given accuracy depends on $\e$ and the phase space box $B$.}\label{fig1}
\end{figure}
For each column of Figure~\ref{fig1} (the width $C$ of the basis functions is fixed here) we compare the two choices $\e=1$ (top) and $\e=0.1$ (bottom).
All panels show that larger phase space boxes yield a worse decay of the error, which is in accordance with Lemma~\ref{fact:truncation}.
In particular, the upper two plots show that for the smallest box (solid lines) the truncation error is reached after approximately 64 grid points (plateaus) and we see that the number of grid points needed to achieve a given tolerance increases with decreasing $\e$, since the small value of $\e$ corresponds to a narrow Gaussian.

\subsection{One-dimensional harmonic oscillator}\label{sub:One-dimensional harmonic oscillator}\hfill\\[2mm]
In this example we consider the quantum harmonic oscillator, which corresponds to the quadratic potential $V(x)=x^2/2$.
For the initial datum we chose the Gaussian wave packet $\psi_0=g_{z_0}^{C_0}$ with $z_0=(1,0)^T, C_0=1i$ and $\e\in\{0.1,1\}$.
In particular, the analytic solution is known to be, see \cite[Theorem~2.5]{Hagedorn:1998},
\begin{align*}
	\psi_{ref}(t)
	=(\pi\e)^{-1/4}\exp\left(-\frac{1}{2\e}\big(x-q(t)\big)^2+\frac{i}{\e}p(t)\big(x-q(t)\big)+\frac{i}{\e}S(t)-\frac{i}{2}t\right),
\end{align*} 
where $q(t),p(t)$ and $S(t)$ are given by
\begin{align*}
	q(t)
	&=q_0\cos(t)+p_0\sin(t),\quad 
	p(t)
	=p_0\cos(t)-q_0\sin(t),\\
	S(t)
	&=-\frac{1}{2}\sin(t)\Big(\big(q_0^2-p_0^2\big)\cos(t)+2q_0p_0\sin(t)\Big).
\end{align*}
The discretization of the wave packet transform was based on the phase space box $B=[-8,8]\times[-8\pi,8\pi]$, where we used 64 grid points in position space, 32 grid points in momentum space and the width parameter $C=4i$ for the basis functions.
The propagation of the basis functions was implemented with the second-order variational splitting integrator in \cite[section~7.5]{Lasser:2020}.
Figure~\ref{fig2} shows the $L^2$-error between the TSTG method and the analytic solution on the spatial interval $B_q=[-8,8]$ for $\e=1$ and two choices of $\tau=0.1$ (red) and $\tau=0.01$ (black).
The step size for the time integration was $h_\tau=1\cdot 10^{-3}$.
\begin{figure}
	\includegraphics[width=\textwidth]{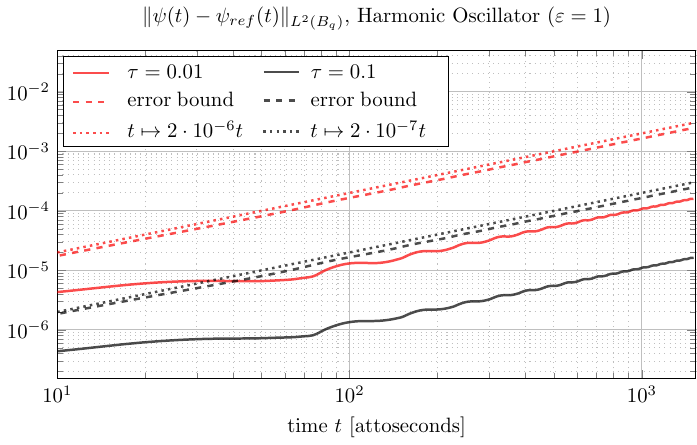}
	\caption{Evolution of the $L^2$-error between the TSTG method and the analytic solution $\psi_{ref}$ for the harmonic oscillator ($\e=1$).
		The errors increase linearly with the number of concatenated steps.
		The time range covers approximately 15 oscillations.
		}\label{fig2}
\end{figure}
The dashed lines indicate the error bound of Theorem~\ref{fact:TSTG_global} based on a direct evaluation of the error bounds $\operatorname{err}_{\mathcal{K}}^{l,\tau}$ in \eqref{eq:error}, where we used again the analytic solution to compute the errors $E_{\mathbf{k}}^\tau$.
We added the linear functions $t\mapsto 2\cdot 10^{-6}t$ (dotted red) and $t\mapsto 2\cdot 10^{-7}t$ (dotted black) to verify that the error increases linearly with the number of TSTG steps.
We note that for $\tau=0.01$ we need 10 times the number of concatenations compared to $\tau= 0.1$ and therefore the slopes of the red and black lines differ by a factor of 10.
To keep the number of TSTG steps and thus the total error small, we recognize that the propagation time $\tau$ should be chosen as large as possible.

Figure~\ref{fig3} shows the $L^2$-error for $\e=0.1$.
\begin{figure}
	\includegraphics[width=\textwidth]{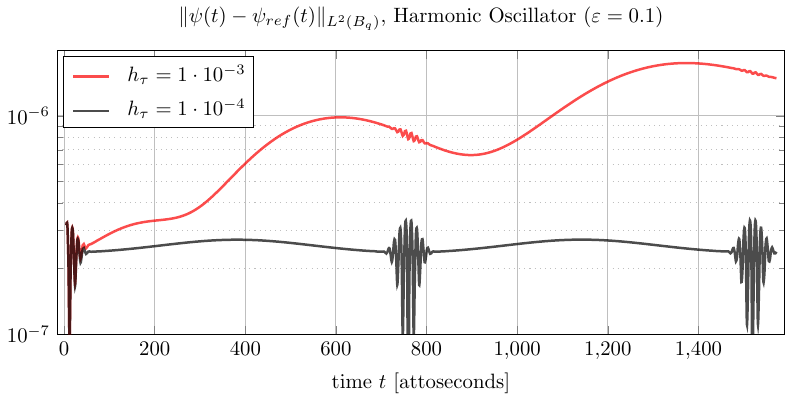}
	\caption{Evolution of the $L^2$-error between the TSTG method and the analytic solution for the harmonic oscillator ($\e=0.1$) for different step sizes $h_\tau$.
		The error increases faster for the coarser time grid (red curve).}\label{fig3}
\end{figure}
Computations were based on 128 grid points in position and momentum space and $\tau=0.01$ for two step sizes $h_\tau=1\cdot10^{-3}$ and $h_\tau=1\cdot10^{-4}$.
For the larger choice of $h_\tau$ (red curve) we see that the error increases faster, which is in accordance with our theoretical result in Proposition~\ref{fact:error_time_discrete}.
For the black curve we can see a periodic pattern (due to the oscillations of the solution) and the linear increase of the error is imperceptible over the time range.
We note that the errors in Figure~\ref{fig3} also show periodic-like oscillations and the linear increase becomes visible because of the long time range (with respect to $\e$).

\subsection{One-dimensional double-well potential}\label{sub:One-dimensional double-well potential}\hfill\\[2mm]
In our last numerical experiment we follow the presentation in \cite[Results]{Kong:2016} by using the one-dimensional double-well potential
\begin{align*}
	V(x)
	=\frac{x^4}{16\eta}-\frac{x^2}{2}
\end{align*}
together with the initial wave function in \eqref{eq:initial_dw} for $\e=1$, which is a model for quantum tunneling.
As for the harmonic oscillator potential, we used again the phase space box $B=[-8,8]\times[-8\pi,8\pi]$ with 64 points in position and momentum space and $C=4i$ for the basis functions.
In addition to the variational Gaussians, we implemented the non-variational Gaussians based on the St\"ormer-Verlet method, see e.g. \cite[chapter~I.1.4]{Hairer:2006}, which have also been used by Kong~\emph{et al.}.
For the reference solution we implemented the split-step Fourier method, using 256 points in the range $B_q=[-8,8]$ with time increment $\tau=0.01$.
The step size $h_\tau=0.001$ was used for both the variational and the non-variational Gaussian propagation.

The upper panels of Figure~\ref{fig4} show the $L^2$-error between the TSTG method and the reference solution for the variational Gaussians (left) and the non-variational Gaussians (right) together with the error bounds of Theorem~\ref{fact:TSTG_global} (dashed lines).
\begin{figure}
	\includegraphics[width=\textwidth]{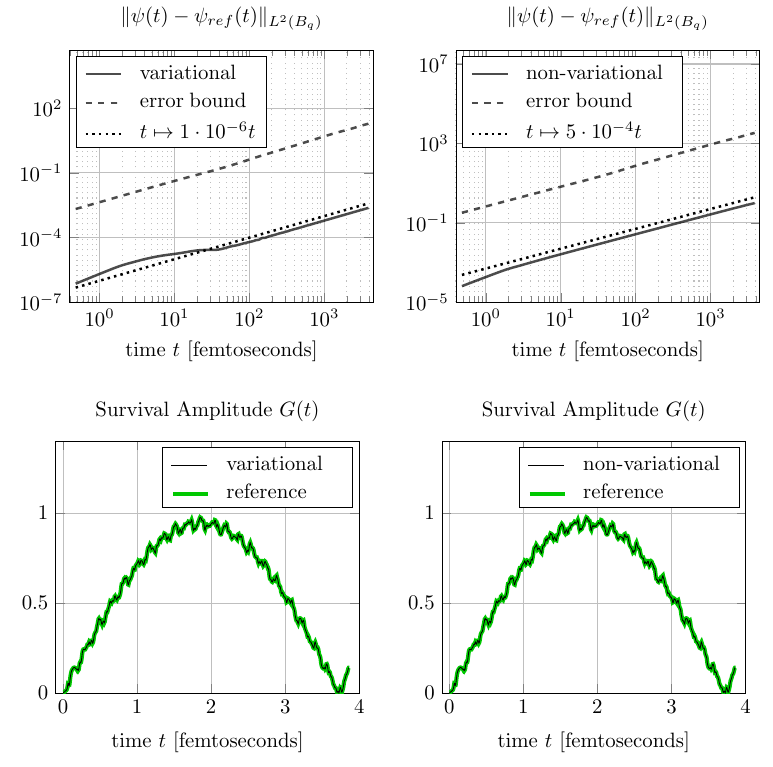}
	\caption{Comparison between variational Gaussians (left) and non-variational Gaussians (right) for the double-well potential.
		Both propagation methods show good agreement compared to benchmark computations.
		Top: Error for the full wave function; Bottom: Survival amplitude;}\label{fig4}
\end{figure}
The lower panels compare the TSTG method with the reference solution for the so-called survival amplitude (overlap between the $\psi(x,t)$ and the mirror image of the initial state on the opposite side of the double-well), which is defined by
\begin{align*}
	G(t)
	:=\int_{\R}\overline{\psi_0(-x)}\psi(x,t)\,\mathrm{d}x
\end{align*}
and is a measure for the tunneling amplitude.\\

The results in Figure~\ref{fig4} show that the TSTG method accurately reproduced the full wave function and the survival amplitude.
The experiments also show that the $L^2$-error increases linearly (approx. as $t\mapsto 10^{-6}t$ for the variational Gaussians), whereas for the non-variational Gaussians the rate is larger (approx. $t\mapsto 5\cdot 10^{-4}t$).
Furthermore, in Figure~\ref{fig5} we compare the TSTG method with the reference solution for the energy expectation values (top) and the relative errors (bottom).
\begin{figure}
	\includegraphics[width=\textwidth]{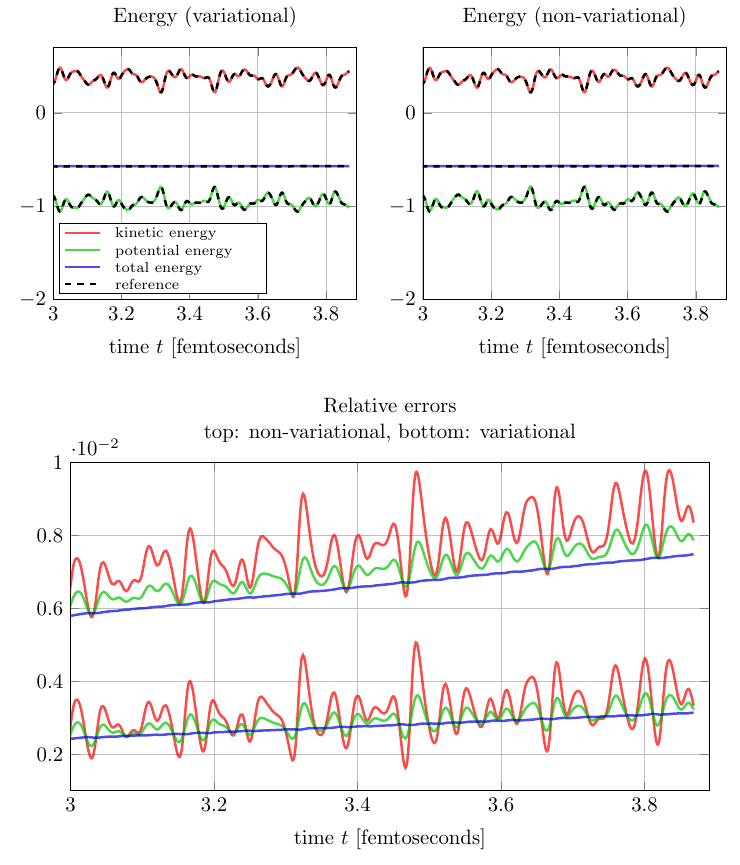}
	\caption{Evolution of energy expectation values (top) and the relative error to the reference solution (bottom) for the variational Gaussians and the non-variational Gaussians between 12,000 and 16,000 TSTG propagation steps.
		}\label{fig5}
\end{figure}
For better illustration we only plotted the time range of the last 4,000 of a total of 16,000 TSTG propagation steps.
We can see that the expectation values of the reference solution are very well approximated even after very long running times.
In particular, the slopes of the blue lines in the lower panel show that the error for the non-variational Gaussians (upper curves) increases faster.

\section{Conclusion and outlook}
In the previous sections we derived a workable error bound for the time-sliced thawed Gaussian propagation method.
The method combines the discretization of the wave packet transform with thawed Gaussian approximations for the propagation of the basis functions.
To provide a mathematical formulation of the TSTG method, we introduced the quadrature-based analysis, synthesis and re-initialization operators $\mathcal{A}_{\mathcal{K}},\mathcal{S}_{\mathcal{K}}$ and $\mathcal{R}^\tau_{\mathcal{K}}$, which allow to write the approximate solution at time $t_n=n\tau$ as
\begin{align*}
	\psi(t_n)
	\approx\psi_{\mathcal{K}}^{n,\tau}
	=\mathcal{S}_{\mathcal{K}}\left(\mathcal{R}_{\mathcal{K}}^\tau\right)^n\mathcal{A}_{\mathcal{K}}\psi_0.
\end{align*}
The algorithm has been implemented in MATLAB to underline our theoretical results and to show that the global error of the method increases linearly with the number $n$ of time steps, regardless of the thawed Gaussian method (variational or non-variational) and the order of the time integrator used.
In the multidimensional setup the method could be improved to a certain extent by using different quadrature rules for the discretization of the wave packet transform.
To make the method applicable especially to high-dimensional systems, the curse of dimensionality must be overcome and the detailed mathematical formulation presented in this paper provides the theoretical fundamentals for combining the method with TT-techniques, which we plan to explore in our future research.

\subsection{Acknowledgments}
Fruitful discussions with Victor S. Batista and Micheline B. Soley are gratefully acknowledged.

\appendix
\section{Analysis and Synthesis Operator}\label{sec:Analysis and Synthesis Operator}
\begin{lemma}
	For $x,y\in\C^{\mathcal{K}}$ and a positive weight $w\in\R^{\mathcal{K}},\,w_{\mathbf{k}}>0$ for all $\mathbf{k}\in\mathcal{K}$, we define the weighted inner product
	\begin{align*}
		\langle x,y\rangle_w
		:=\sum_{\mathbf{k}\in\mathcal{K}}\overline{x_{\mathbf{k}}}\,y_{\mathbf{k}}\,w_{\mathbf{k}}^{-1}.
	\end{align*}
	Moreover, for a given phase space grid $\{z_{\mathbf{k}}\}_{\mathbf{k}\in\mathcal{K}}$ let $\mathcal{A}_{\mathcal{K}}$ be the operator that maps a square-integrable function $\psi\in L^2(\R^d)$ to the coefficient tensor $(c_{\mathbf{k}}(\psi))$, as well as $\mathcal{S}_{\mathcal{K}}$ the corresponding synthesis operator.
	Then,
	\begin{align*}
		\langle\mathcal{S}_{\mathcal{K}}c\mid\psi\rangle_{L^2(\R^d)}
		=\langle c,\mathcal{A}_{\mathcal{K}}\psi\rangle_w
		\quad\text{for all $\psi\in L^2(\R^d)$ and $c\in\C^{\mathcal{K}}$.}
	\end{align*}
\end{lemma}
\begin{proof}
	Let $c\in\C^{\mathcal{K}}$ and $\psi\in L^2(\R^d)$.
	By definition of the synthesis operator we have
	\begin{align*}
		\langle\mathcal{S}_{\mathcal{K}}c\mid\psi\rangle_{L^2(\R^d)}
		&=\int_{\R^d}\sum_{\mathbf{k}\in\mathcal{K}}\overline{c_{\mathbf{k}}\,g_{\mathbf{k},0}(x)}\psi(x)\,\mathrm{d}x\\
		&=\sum_{\mathbf{k}\in\mathcal{K}}\overline{c_{\mathbf{k}}}\,w_{\mathbf{k}}^{-1}\,w_{\mathbf{k}}\langle g_{\mathbf{k},0}\mid\psi\rangle_{L^2(\R^d)}
		=\langle c,\mathcal{A}_{\mathcal{K}}\psi\rangle_w.
	\end{align*}
\end{proof}
%

\section{Inner Products of Gaussians}\label{sec:Inner Products of Gaussians_appendix}
\begin{proof}[Proof (of Lemma~\ref{fact:inner})]
	The product of the functions $\overline{g_{z_1}^{C_1,\e}}$ and $g_{z_2}^{C_2,\e}$ is a Gaussian.
	To obtain an explicit representation, we rewrite the sum of the exponents
	\begin{align*}
		-\frac{i}{\e}\left(\frac{1}{2}(x-q_1)^T\bar{C_1}(x-q_1)+p_1^T(x-q_1)\right)+\frac{i}{\e}\left(\frac{1}{2}(x-q_2)^TC_2(x-q_2)+p_2^T(x-q_2)\right)
	\end{align*}
	as a quadratic function
	\begin{align*}
		\frac{i}{\e}\left(\frac{1}{2}(x-q_2)^TB(x-q_2)+(x-q_2)^Tb+c\right),
	\end{align*}
	where a short calculation shows that $B\in\C^{d\times d},\,b\in\C^d$ and $c\in\C$ are given by
	\begin{equation}\label{eq:inner_parametersBbc}
		\begin{split}
			B
			&:=C_2-\bar C_1,\quad
			b
			:=(p_2-p_1)-\bar C_1(q_2-q_1)
			\quad\text{and}\\
			c
			&:=-\frac{1}{2}(q_2-q_1)^T\bar C_1(q_2-q_1)-p_1^T(q_2-q_1).
		\end{split}
	\end{equation}
	In particular, since $\operatorname{Im}(-\bar C_1)=\operatorname{Im}(C_1)$ is positive definite and the sum of two real positive definite matrices is again positive definite, we conclude that $B$ is an element of the Siegel space $\mathfrak{S}^+(d)$.
	This yields the following representation for all $x\in\R^d$:
	\begin{align*}
		\overline{g_{z_1}^{C_1,\e}(x)}g_{z_2}^{C_2,\e}(x)
		=\alpha\exp\left[\frac{i}{\e}\left(\frac{1}{2}(x-q_2)^TB(x-q_2)+(x-q_2)^Tb+c\right)\right],
	\end{align*}
	where the positive constant $\alpha>0$ is given by
	\begin{align*}
		\alpha
		&:=(\pi\e)^{-d/2}\det(\operatorname{Im}C_1\operatorname{Im}C_2)^{1/4}.
	\end{align*}
	Therefore, we conclude that
	\begin{align*}
		\langle g_{z_1}^{C_1,\e},g_{z_2}^{C_2,\e}\rangle_{L^2(\R^d)}
		&=\int_{\R^d}\alpha\exp\left[\frac{i}{\e}\left(\frac{1}{2}(x-q_2)^TB(x-q_2)+(x-q_2)^Tb+c\right)\right]\,\mathrm{d}x\\
		&=\alpha\int_{\R^d}\exp\left[\frac{i}{\e}\left(\frac{1}{2}y^TBy+y^Tb+c\right)\right]\,\mathrm{d}y,
	\end{align*}
	where a formula of multivariate Gaussian integrals (see e.g. \cite[Appendix~A, Theorem~1]{Folland:1989}) yields
	\begin{align*}
		\int_{\R^d}\exp\left[\frac{i}{\e}\left(\frac{1}{2}y^TBy+y^Tb+c\right)\right]\,\mathrm{d}y
		=\frac{(2\pi\e)^{d/2}}{\sqrt{\det(-iB)}}\exp\left(-\frac{i}{2\e}b^TB^{-1}b+\frac{i}{\e}c\right).
	\end{align*}
	We note that the branch of the square root is determined by the requirement
	\begin{align*}
		\det(-iB)^{-1/2}
		>0
	\end{align*}
	if $-iB$ is real and positive definite.
	Moreover, using the formulas in equation \eqref{eq:inner_parametersBbc}, we obtain the following representation:
	\begin{align*}
		&\hspace{35mm}\alpha\frac{(2\pi\e)^{d/2}}{\sqrt{\det(-iB)}}\exp\left(-\frac{i}{2\e}b^TB^{-1}b+\frac{i}{\e}c\right)\\
		&=\frac{2^{d/2}\det(\operatorname{Im}C_1\operatorname{Im}C_2)^{1/4}}{\sqrt{\det(-iB)}}\exp\left(\frac{i}{2\e}(p_1+p_2)^T(q_1-q_2)\right)\cdots\\
		&\qquad\exp\left(\frac{i}{2\e}(p_2-p_1)^T(C_2-\bar C_1)^{-1}(C_2+\bar C_1)(q_2-q_1)\right)\cdots\\
		&\qquad\exp\left(\frac{i}{2\e}(p_2-p_1)^T(-B^{-1})(p_2-p_1)\right)\exp\left(\frac{i}{2\e}(q_2-q_1)^T(-\bar C_1-\bar C_1B^{-1}\bar C_1)(q_2-q_1)\right).
	\end{align*}
	In the last line we have two Gaussians: One with respect to the difference $p_2-p_1$ with width matrix $-B^{-1}$ and one for $q_2-q_1$ with width matrix $-\bar C_1-\bar C_1B^{-1}\bar C_1$.
	In particular, the Woodbury matrix identity, see e.g. \cite[Page~258]{Higham:2002}, yields
	\begin{align*}
		-\bar C_1-\bar C_1B^{-1}\bar C_1
		=\left(C_2^{-1}-\bar C_1^{-1}\right)^{-1}.
	\end{align*}
	Hence, since $Z\in\mathfrak{S}^+(d)$ implies $-Z^{-1}\in\mathfrak{S}^+(d)$ (see e.g. \cite[Theorem~4.64]{Folland:1989}), we conclude that both width matrices
	\begin{align*}
		-B^{-1}
		\quad\text{and}\quad
		\left(C_2^{-1}-\bar C_1^{-1}\right)^{-1}
	\end{align*}
	are in $\mathfrak{S}^+(d)$ and therefore we conclude that the block diagonal matrix $M$ in \eqref{eq:inner2} is an element of $\mathfrak{S}^+(2d)$.
	Putting together the above calculations we arrive at \eqref{eq:inner1}.\\
	
	To prove the bound in \eqref{eq:inner3}, we follow the idea of \cite[11.4~Lemma]{Swart:2008} and assume that the eigenvalues of $\operatorname{Im}(C_k)$ and $\operatorname{Im}(-C_k^{-1})$ are bounded from below by $\theta>0$ and from above by $\Theta>0$.
	Furthermore, let us introduce the real-valued Gaussian function
	\begin{align*}
		g_k^{\theta}(x)
		=(\pi\e)^{-d/4}\theta^{d/4}\exp\left(-\frac{\theta}{2\e}\|x-q_k\|_2^2\right),
		\quad k=1,2,\,x\in\R^d.
	\end{align*}
	Then, for all $x\in\R^d$, the spectral bounds imply that
	\begin{align*}
		|g_{z_k}^{C_k,\e}(x)|
		\le\det(\operatorname{Im}C_k)^{1/4}\theta^{-d/4}g_k^{\theta}(x)
		\le\Theta^{d/4}\theta^{-d/4}g_k^{\theta}(x),
	\end{align*}
	and therefore we obtain the following bound:
	\begin{equation}\label{eq:inner_part1}
		\begin{split}
			\left|\langle g_{z_1}^{C_1,\e},g_{z_2}^{C_2,\e}\rangle_{L^2(\R^d)}\right|
			&\le\theta^{-d/2}\Theta^{d/2}\langle g_1^{\theta},g_2^{\theta}\rangle_{L^2(\R^d)}\\
			&=\theta^{-d/2}\Theta^{d/2}\exp\left(-\frac{\theta}{4\e}\|q_2-q_1\|_2^2\right),
		\end{split}
	\end{equation}
	where the last equality follows by the formula in \eqref{eq:inner1}.
	Furthermore, combining Plancherel's theorem for the $\e$-rescaled Fourier transform $\F_\e\colon L^2(\R^d)\to L^2(\R^d)$, defined for all $p\in\R^d$ by
	\begin{align*}
		\F_\e\psi(p)
		:=(2\pi\e)^{-d/2}\int_{\R^d}\psi(x)e^{-ip\cdot x/\e}\,\mathrm{d}x,
	\end{align*}
	with a formula for the Fourier transform $\F_\e g_{z_k}^{C_k,\e}$, implies
	\begin{equation}\label{eq:inner_part2}
		\begin{split}
			\left|\langle g_{z_1}^{C_1,\e},g_{z_2}^{C_2,\e}\rangle_{L^2(\R^d)}\right|
			&=\left|\langle\F_\e g_{z_1}^{C_1,\e},\F_\e g_{z_2}^{C_2,\e}\rangle_{L^2(\R^d)}\right|\\
			&\le\theta^{-d/2}\Theta^{d/2}\exp\left(-\frac{\theta}{4\e}\|p_2-p_1\|_2^2\right).
		\end{split}
	\end{equation}
	Consequently, combining the bounds in \eqref{eq:inner_part1} and \eqref{eq:inner_part2} proves \eqref{eq:inner3} for
	\begin{align*}
		\zeta
		=\left(\frac{\Theta}{\theta}\right)^d.
	\end{align*}
\end{proof}
%

\section{Discrete Gaussian Convolution}\label{sec:Discrete Gaussian Convolution}
\begin{lemma}
	For $\sigma>0$ consider the one-dimensional Gaussian function
	\begin{align*}
		f_\sigma(t)
		:=\exp\left(-\frac{1}{2\sigma}t^2\right)
		\quad\text{for all $t\in\R$}.
	\end{align*}
	For arbitrary grid points $t_1<t_2<...<t_N$ let
	\begin{align}\label{eq:convolution1}
		h_i
		:=t_{i+1}-t_i,\,
		i
		=1,\dots,N-1
		\quad\text{and}\quad
		h
		:=\min_{i=1,\dots,N-1}h_i.
	\end{align}
	Then, for all $\sigma_1,\sigma_2>0$, there exists a constant $c>0$ such that for all $s\in\R$ we have
	\begin{align}\label{eq:convolution2}
		\sum_{k=1}^Nf_{\sigma_1}(t_k)f_{\sigma_2}(s-t_k)
		\le cf_{\sigma_1+\sigma_2}(s),
	\end{align}	
	where $c$ depends on $\sigma_1,\sigma_2$ and $h$, but not on $N$.
\end{lemma}
\begin{proof}
	Let $s\in\R$.
	A short calculation shows that
	\begin{align*}
		f_{\sigma_1}(t_k)f_{\sigma_2}(s-t_k)
		=f_{\sigma_1+\sigma_2}(s)f_{\sigma_3}\left(s'-t_k\right),
	\end{align*}
	where we introduced the parameters
	\begin{align*}
		\sigma_3
		=\frac{\sigma_1\sigma_2}{\sigma_1+\sigma_2}
		\quad\text{and}\quad
		s'
		=\frac{\sigma_1s}{\sigma_1+\sigma_2}.
	\end{align*}
	Consequently, the sum in \eqref{eq:convolution2} can be written as
	\begin{align*}
		\sum_{k=1}^Nf_{\sigma_1}(t_k)f_{\sigma_2}(s-t_k)
		=f_{\sigma_1+\sigma_2}(s)\sum_{k=1}^Nf_{\sigma_3}\left(s'-t_k\right),
	\end{align*}
	where the sum at the right hand-side can be bounded independently of $s'$ as
	\begin{align*}
		\sum_{k=1}^Nf_{\sigma_3}\left(s'-t_k\right)
		\le\sum_{k\in\Z}f_{\sigma_3}\left(hk\right),
	\end{align*}
	where the minimal distance $h$ between consecutive grid points is defined in \eqref{eq:convolution1}.
	In particular, since the last sum can be viewed as a Riemann sum approximation to the integral
	\begin{align*}
		\frac{1}{h}\int_\R f_{\sigma_3}\left(t\right)\mathrm{d}t
		=\frac{\sqrt{2\pi\sigma_3}}{h},
	\end{align*}
	we conclude that there exists a positive constant $c>0$, depending on $\sigma_3$ and $h$, such that
	\begin{align*}
		\sum_{k\in\Z}f_{\sigma_3}\left(s'-t_k\right)
		\le c,
	\end{align*}
	which makes the proof complete.
\end{proof}
%

\section{Reconstruction Error}\label{sec:Reconstruction Error}
In the following, we prove that for all $\mathbf{k}\in\mathcal{K}$ there exists a positive constant $\tilde c_{\mathbf{k}}>0$ such that
\begin{align}\label{eq:ewp_bound_app}
	E_{wp}(u^\tau_{\mathbf{k}})
	\le\tilde c_{\mathbf{k}}\left(\frac{1+h_\tau^s}{\e}\right)\tau.
\end{align}
Therefore, let us fix $\mathbf{k}\in\mathcal{K}$ and decompose the Gaussian $u^\tau_{\mathbf{k}}$, which is the numerical approximation to the time-evolved Gaussian basis function $g_{\mathbf{k}}(\tau)$ after a short TSTG propagation time $\tau>0$, as follows:
\begin{align*}
	u^\tau_{\mathbf{k}}
	=g_{\mathbf{k},0}+\Big(u^\tau_{\mathbf{k}}-u_{\mathbf{k}}(\tau)\Big)+\Big(u_{\mathbf{k}}(\tau)-g_{\mathbf{k},0}\Big)
	=:g_{\mathbf{k},0}+R^1_\mathbf{k}(\tau)+R^2_\mathbf{k}(\tau),
\end{align*}
where $g_{\mathbf{k},0}$ is the initial basis function and $u_{\mathbf{k}}(\tau)\in\mathcal{M}$ the approximation in the manifold of complex Gaussians.
Using that $g_{\mathbf{k},0}$ is an element of $\mathcal{V}_{\mathcal{K}}$, the definition of the discretization error in \eqref{eq:error_wp} yields that $E_{wp}(g_{\mathbf{k},0})=0$ and therefore
\begin{align*}
	E_{wp}(u^\tau_{\mathbf{k}})
	\le E_{wp}(R^1_\mathbf{k}(\tau))+E_{wp}(R^2_\mathbf{k}(\tau)).
\end{align*}
Hence, it suffices to show that $E_{wp}(R^j_\mathbf{k}(\tau))=\mathcal{O}(\tau)$.
Firstly, we see that
\begin{align}\label{eq:wpR1}
	E_{wp}(R^1_\mathbf{k}(\tau))
	\le\|u^\tau_{\mathbf{k}}-u_{\mathbf{k}}(\tau)\|\left(1+\|\mathcal{A}_{\mathcal{K}}^*\mathcal{A}_{\mathcal{K}}\|\right),
\end{align}
where $\|\mathcal{A}_{\mathcal{K}}^*\mathcal{A}_{\mathcal{K}}\|$ denotes the operator norm of the linear operator $\mathcal{A}_{\mathcal{K}}^*\mathcal{A}_{\mathcal{K}}$.
In particular, a short calculation shows that
\begin{align*}
	\|\mathcal{A}_{\mathcal{K}}^*\mathcal{A}_{\mathcal{K}}\|
	\le\sum_{\mathbf{k}\in\mathcal{K}}w_{\mathbf{k}}
	:=W_\mathcal{K},
\end{align*}
where $w_{\mathbf{k}}\ge 0$ are the weights of the underlying quadrature rule.
Hence, combining \eqref{eq:wpR1} with the bound for $\|u^\tau_{\mathbf{k}}-u_{\mathbf{k}}(\tau)\|$ in \eqref{eq:thawed_integrator}, we conclude that
\begin{align*}
	E_{wp}(R^1_\mathbf{k}(\tau))
	\le c_{\mathbf{k}}^{(2)}\tau\frac{h_\tau^s}{\e}\left(1+W_\mathcal{K}\right).
\end{align*}
Finally, using that $u_{\mathbf{k}}(t)\in\mathcal{M}$ is the exact solution to
\begin{align*}
	i\e\partial_t u_{\mathbf{k}}(t)
	=-\frac{\e^2}{2}\Delta_x u_{\mathbf{k}}(t)+U_{q_{\mathbf{k}}(t)}u_{\mathbf{k}}(t),
	\quad
	u_{\mathbf{k}}(0)
	=g_{\mathbf{k},0},
\end{align*}
where $U_{q_{\mathbf{k}}}$ denotes the second-order Taylor polynomial of $V$ at $q_{\mathbf{k}}$, we conclude that $\|u_{\mathbf{k}}(\tau)-g_{\mathbf{k},0}\|$ can be estimated in terms of the time-dependent Hamiltonian
\begin{align*}
	H_{\mathbf{k}}(t)
	:=-\frac{\e^2}{2}\Delta_x+U_{q_{\mathbf{k}}(t)}.
\end{align*}
We have
\begin{align*}
	u_{\mathbf{k}}(\tau)-g_{\mathbf{k},0}
	=\int_0^\tau\dot u_{\mathbf{k}}(s)\,\mathrm{d}s
	=\frac{1}{i\e}\int_0^\tau H_{\mathbf{k}}(s)u_{\mathbf{k}}(s)\,\mathrm{d}s
\end{align*}
and a short calculation shows that
\begin{align*}
	-\frac{\e^2}{2}\Delta_x u_{\mathbf{k}}(s)
	=&\Bigg(\frac{1}{2}(x-q_{\mathbf{k}}(s))^TC_{\mathbf{k}}(s)^2(x-q_{\mathbf{k}}(s))+p_{\mathbf{k}}(s)^TC_{\mathbf{k}}(s)(x-q_{\mathbf{k}}(s))\dots\\
	&\qquad+\,\frac{1}{2}|p_{\mathbf{k}}(s)|^2-\frac{i\e}{2}\operatorname{tr}C_{\mathbf{k}}(s)\Bigg)u_{\mathbf{k}}(s).
\end{align*}
Hence, using an estimate for moments of Gaussians, see e.g. \cite[Lemma~3.8]{Lasser:2020}, we obtain
\begin{align*}
	\left\|-\frac{\e^2}{2}\Delta_x u_{\mathbf{k}}(s)\right\|
	=\frac{1}{2}|p_{\mathbf{k}}(s)|^2+\mathcal{O}(\sqrt{\e}),
\end{align*}
and similarly
\begin{align*}
	\left\|U_{q_{\mathbf{k}}(s)}u_{\mathbf{k}}(s)\right\|
	=|V(q_{\mathbf{k}}(s))|+\mathcal{O}(\sqrt{\e}).
\end{align*}
Altogether,
\begin{align*}
	\|u_{\mathbf{k}}(\tau)-g_{\mathbf{k},0}\|
	\le\frac{\tau}{\e}\left(\sup_{s\in[0,\tau]}\left(\frac{1}{2}|p_{\mathbf{k}}(s)|^2+|V(q_{\mathbf{k}}(s))|\right)+\rho_{\mathbf{k}}(\tau)\right)
	=:\frac{\tau}{\e}C_{\mathbf{k}},
\end{align*}
with $\rho_{\mathbf{k}}(\tau)=\mathcal{O}(\sqrt\e)$ uniformly in $\tau$.
This shows that
\begin{align*}
	E_{wp}(R^2_\mathbf{k}(\tau))
	\le\frac{\tau}{\e}C_{\mathbf{k}}\left(1+W_\mathcal{K}\right)
\end{align*}
and therefore the bound in \eqref{eq:ewp_bound_app} follows for
\begin{align*}
	\tilde c_{\mathbf{k}}
	=\max\left(c_{\mathbf{k}}^{(2)},C_{\mathbf{k}}\right)\left(1+\sum_{\mathbf{k}\in\mathcal{K}}w_{\mathbf{k}}\right).
\end{align*}


\begin{thebibliography}{MMC90}

\bibitem[BG20]{Blanes:2020}
S.~Blanes and V.~Gradinaru.
\newblock {High order efficient splittings for the semiclassical time-dependent
  Schr{\"o}dinger equation}.
\newblock {\em Journal of Computational Physics}, 405:109157, 2020.

\bibitem[BL20]{Bergold:2020u}
P.~Bergold and C.~Lasser.
\newblock {The Gaussian Wave Packet Transform via Quadrature Rules}.
\newblock Preprint on arXiv, {https://arxiv.org/abs/2010.03478}, 2020.

\bibitem[CDS03]{Chiani:2003}
M.~Chiani, D.~Dardari, and M.~K. Simon.
\newblock New exponential bounds and approximations for the computation of
  error probability in fading channels.
\newblock {\em IEEE Transactions on Wireless Communications}, 2(4):840--845,
  2003.

\bibitem[CK90]{Coalson:1990}
R.~D. Coalson and M.~Karplus.
\newblock {Multidimensional variational Gaussian wave packet dynamics with
  application to photodissociation spectroscopy}.
\newblock {\em The Journal of Chemical Physics}, 93(6):3919--3930, 1990.

\bibitem[CR12]{Combescure:2012}
M.~Combescure and D.~Robert.
\newblock {\em {Coherent States and Applications in Mathematical Physics}}.
\newblock Theoretical and Mathematical Physics. Springer Cham, {Second}
  edition, 2012.

\bibitem[DR07]{Davis:2007}
P.~J. Davis and P.~Rabinowitz.
\newblock {\em {Methods of Numerical Integration}}.
\newblock Academic Press, {Second} edition, 2007.

\bibitem[DT10]{Descombes:2010}
S.~Descombes and M.~Thalhammer.
\newblock {An exact local error representation of exponential operator
  splitting methods for evolutionary problems and applications to linear
  Schr{\"o}dinger equations in the semi-classical regime}.
\newblock {\em BIT Numerical Mathematics}, 50:729--749, 2010.

\bibitem[FF15]{Fornberg:2015}
B.~Fornberg and N.~Flyer.
\newblock {Solving PDEs with radial basis functions}.
\newblock {\em Acta Numerica}, 24:215--258, 2015.

\bibitem[FGL09]{Faou:2009}
E.~Faou, V.~Gradinaru, and C.~Lubich.
\newblock {Computing Semiclassical Quantum Dynamics with Hagedorn Wavepackets}.
\newblock {\em SIAM Journal on Scientific Computing}, 31(4):3027--3041, 2009.

\bibitem[FL06]{Faou:2006}
E.~Faou and C.~Lubich.
\newblock {A Poisson Integrator for Gaussian Wavepacket Dynamics}.
\newblock {\em Computing and Visualization in Science}, 9(2):45--55, 2006.

\bibitem[FLF11]{Fornberg:2011}
B.~Fornberg, E.~Larsson, and N.~Flyer.
\newblock {Stable Computations with Gaussian Radial Basis Functions}.
\newblock {\em SIAM Journal on Scientific Computing}, 33(2):869--892, 2011.

\bibitem[Fol89]{Folland:1989}
G.~B. Folland.
\newblock {\em {Harmonic Analysis in Phase Space}}.
\newblock Annals of Mathematics Studies. Princeton University Press, 1989.

\bibitem[FS98]{Feichtinger:1998}
H.~G. Feichtinger and T.~Strohmer.
\newblock {\em {Gabor Analysis and Algorithms: Theory and Applications}}.
\newblock Applied and Numerical Harmonic Analysis. Springer Science \& Business
  Media, 1998.

\bibitem[GG98]{Gerstner:1998}
T.~Gerstner and M.~Griebel.
\newblock {Numerical integration using sparse grids}.
\newblock {\em Numerical Algorithms}, 18(3):209--232, 1998.

\bibitem[GH14]{Gradinaru:2014}
V.~Gradinaru and G.~A. Hagedorn.
\newblock {Convergence of a semiclassical wavepacket based time-splitting for
  the Schr{\"o}dinger equation}.
\newblock {\em Numerische Mathematik}, 126(1):53--73, 2014.

\bibitem[Gr{\"o}01]{Grochenig:2001}
K.~Gr{\"o}chenig.
\newblock {\em {Foundations of Time-Frequency Analysis}}.
\newblock Applied and Numerical Harmonic Analysis. Springer Science \& Business
  Media, 2001.

\bibitem[Hag80]{Hagedorn:1980}
G.~A. Hagedorn.
\newblock {Semiclassical quantum mechanics. I. The {$\hbar \rightarrow 0$}
  limit for coherent states}.
\newblock {\em Communications in Mathematical Physics}, 71(1):77--93, 1980.

\bibitem[Hag98]{Hagedorn:1998}
G.~A. Hagedorn.
\newblock {Raising and Lowering Operators for Semiclassical Wave Packets}.
\newblock {\em Annals of Physics}, 269(1):77--104, 1998.

\bibitem[Hel75]{Heller:1975}
E.~J. Heller.
\newblock {Time‐dependent approach to semiclassical dynamics}.
\newblock {\em The Journal of Chemical Physics}, 62(4):1544--1555, 1975.

\bibitem[Hel76]{Heller:1976}
E.~J. Heller.
\newblock {Time dependent variational approach to semiclassical dynamics}.
\newblock {\em The Journal of Chemical Physics}, 64(1):63--73, 1976.

\bibitem[Hel81]{Heller:1981}
E.~J. Heller.
\newblock {Frozen Gaussians: A very simple semiclassical approximation}.
\newblock {\em The Journal of Chemical Physics}, 75(6):2923--2931, 1981.

\bibitem[Hig02]{Higham:2002}
N.~J. Higham.
\newblock {\em {Accuracy and Stability of Numerical Algorithms}}.
\newblock Other Titles in Applied Mathematics. Society for Industrial and
  Applied Mathematics (SIAM), {Second} edition, 2002.

\bibitem[HK84]{Herman:1984}
M.~F. Herman and E.~Kluk.
\newblock {A semiclasical justification for the use of non-spreading
  wavepackets in dynamics calculations}.
\newblock {\em Chemical Physics}, 91(1):27--34, 1984.

\bibitem[HLW06]{Hairer:2006}
E.~Hairer, C.~Lubich, and G.~Wanner.
\newblock {\em {Geometric Numerical Integration: Structure-Preserving
  Algorithms for Ordinary Differential Equations}}.
\newblock Springer Series in Computational Mathematics. Springer Berlin,
  Heidelberg, {Second} edition, 2006.

\bibitem[KLY19]{Kormann:2019}
K.~Kormann, C.~Lasser, and A.~Yurova.
\newblock {Stable Interpolation with Isotropic and Anisotropic Gaussians Using
  Hermite Generating Function}.
\newblock {\em SIAM Journal on Scientific Computing}, 41(6):A3839--A3859, 2019.

\bibitem[KMB16]{Kong:2016}
X.~Kong, A.~Markmann, and V.~S. Batista.
\newblock {Time-Sliced Thawed Gaussian Propagation Method for Simulations of
  Quantum Dynamics}.
\newblock {\em The Journal of Physical Chemistry A}, 120(19):3260--3269, 2016.

\bibitem[LL20]{Lasser:2020}
C.~Lasser and C.~Lubich.
\newblock {Computing quantum dynamics in the semiclassical regime}.
\newblock {\em Acta Numerica}, 29:229--401, 2020.

\bibitem[LQ09]{Leung:2009}
S.~Leung and J.~Qian.
\newblock {Eulerian Gaussian beams for Schr{\"o}dinger equations in the
  semi-classical regime}.
\newblock {\em Communications in Computational Physics}, 228(8):2951--2977,
  2009.

\bibitem[LRT13]{Liu:2013}
H.~Liu, O.~Runborg, and N.~M. Tanushev.
\newblock {Error Estimates for Gaussian Beam Superpositions}.
\newblock {\em Mathematics of Computation}, 82(282):919--952, 2013.

\bibitem[LS17]{Lasser:2017}
C.~Lasser and D.~Sattlegger.
\newblock {Discretising the Herman--Kluk Propagator}.
\newblock {\em Numerische Mathematik}, 137(1):119--157, 2017.

\bibitem[Lub08]{Lubich:2008}
C.~Lubich.
\newblock {\em {From Quantum to Classical Molecular Dynamics: Reduced Models
  and Numerical Analysis}}.
\newblock Zurich Lectures in Advanced Mathematics. European Mathematical
  Society (EMS), 2008.

\bibitem[Mar02]{Martinez:2002}
A.~Martinez.
\newblock {\em {An Introduction to Semiclassical and Microlocal Analysis}}.
\newblock Universitext. Springer New York, 2002.

\bibitem[MMC90]{Meyer:1990}
H.-D. Meyer, U.~Manthe, and L.S. Cederbaum.
\newblock {The multi-configurational time-dependent Hartree approach}.
\newblock {\em Chemical Physics Letters}, 165(1):73--78, 1990.

\bibitem[MQ02]{McLachlan:2002}
R.~I. McLachlan and G.~R.~W. Quispel.
\newblock {Splitting methods}.
\newblock {\em Acta Numerica}, 11:341--434, 2002.

\bibitem[Ose11]{Oseledets:2011}
I.~V. Oseledets.
\newblock {Tensor-Train Decomposition}.
\newblock {\em SIAM Journal on Scientific Computing}, 33(5):2295--2317, 2011.

\bibitem[OT09]{Oseledets:2009}
I.~V. Oseledets and E.~E. Tyrtyshnikov.
\newblock {Breaking the Curse of Dimensionality, Or How to Use SVD in Many
  Dimensions}.
\newblock {\em SIAM Journal on Scientific Computing}, 31(5):3744--3759, 2009.

\bibitem[Sie39]{Siegel:1939}
C.~L. Siegel.
\newblock {Einf{\"u}hrung in die Theorie der Modulfunktionen $n$-ten Grades}.
\newblock {\em Mathematische Annalen}, 116(1):617--657, 1939.

\bibitem[Swa08]{Swart:2008}
T.~C. Swart.
\newblock {\em {Initial Value Representations}}.
\newblock Dissertation, Freie Universit{\"a}t Berlin, 2008.

\bibitem[WRB04]{Worth:2004}
G.~A. Worth, M.~A. Robb, and I.~Burghardt.
\newblock {A novel algorithm for non-adiabatic direct dynamics using
  variational Gaussian wavepackets}.
\newblock {\em Faraday Discussions}, 127:307--323, 2004.

\bibitem[Zhe14]{Zheng:2014}
C.~Zheng.
\newblock {Optimal Error Estimates for First-Order Gaussian Beam Approximations
  to the Schr{\"o}dinger Equation}.
\newblock {\em SIAM Journal on Numerical Analysis}, 52(6):2905--2930, 2014.

\end{thebibliography}
\end{document}